\documentclass[onecolumn, 11pt]{IEEEtran}
\usepackage{amsmath,amsfonts,amssymb,euscript, graphicx, 
epsfig,
enumerate,float,afterpage, subfigure, ifthen}%

\newtheorem{thm}{Theorem}

\newtheorem{lem}{Lemma}

\newtheorem{exercise}{Exercise}

\newcommand{\expect}[1]{\mathbb{E}\left[#1\right]}
\newcommand{\defequiv}{\mbox{\raisebox{-.3ex}{$\overset{\vartriangle}{=}$}}}
\newcommand{\norm}[1]{||{#1}||}
\newcommand{\bv}[1]{{\boldsymbol{#1} }}
\newcommand{\script}[1]{{{\cal{#1} }}}

\begin{document}

\title
  {Low Power Dynamic Scheduling for Computing Systems}
\author{Michael J. Neely%
\thanks{The author is with the  Electrical Engineering department at the University
of Southern California, Los Angeles, CA.} 
\thanks{This material is supported in part  by one or more of 
the following:  the NSF Career grant CCF-0747525, the 
Network Science Collaborative Technology Alliance sponsored
by the U.S. Army Research Laboratory W911NF-09-2-0053.}}

\markboth{}{Neely}

\maketitle

\begin{abstract}   
This paper considers energy-aware control for a computing system with 
two states: \emph{active} and \emph{idle}.  In the active state, the controller chooses to perform
a single task using one of multiple task processing modes.  The controller then
saves energy by  
choosing an amount of time for the system to be idle. These decisions affect 
processing time, energy expenditure, and an abstract
\emph{attribute vector} that can be used to model other criteria of 
interest (such as processing quality or distortion). 
The goal is to optimize time average system performance.   
Applications of this model include a smart phone that makes energy-efficient computation 
and transmission decisions, a computer that processes tasks subject to rate, quality, and power
constraints, and a smart grid energy manager that allocates resources in reaction 
to a time varying energy price.  
The solution methodology of this paper uses the theory of \emph{optimization for renewal systems} developed
in our previous work.  This paper is written in tutorial form and develops the main 
concepts of the theory using several detailed examples. It also highlights the relationship between 
online dynamic optimization and linear fractional programming.  Finally, it 
provides exercises to help the reader learn the 
main concepts and apply them to their own optimizations.   
This paper is an arxiv technical report, and 
is a preliminary version of material that will appear as a book chapter in an upcoming
book on green communications and networking.  
\end{abstract} 

\begin{keywords} 
Queueing analysis, optimization, stochastic control, renewal theory
\end{keywords}

\section{Introduction} 

This paper considers energy-aware control for a computing system with 
two states: \emph{active} and \emph{idle}.  In the active state, the controller chooses to perform
a single task using one of multiple task processing modes.  The controller then
saves energy by  
choosing an amount of time for the system to be idle. These decisions affect 
processing time, energy expenditure, and an abstract
\emph{attribute vector} that can be used to model other criteria of 
interest (such as processing quality or distortion). 
The goal is to optimize time average system performance.   
Applications of this model include a smart phone that makes energy-efficient computation 
and transmission decisions, a computer that processes tasks subject to rate, quality, and power
constraints, and a smart grid energy manager that allocates resources in reaction 
to a time varying energy price.  

The solution methodology of this paper uses the theory of \emph{optimization for renewal systems} developed
in \cite{sno-text}\cite{renewals-asilomar2010}. Section \ref{section:task-scheduling}  
focuses on a computer system that seeks to minimize average power subject 
to processing rate constraints for different classes of tasks.  Section 
\ref{section:general-attributes} generalizes to treat optimization for a larger class of systems.
Section \ref{section:random-events} extends the model to allow control actions
to react to a random event observed at the beginning of each active period, such as a vector
of current channel conditions or energy prices. 

\section{Task Scheduling with Processing Rate Constraints} \label{section:task-scheduling} 

\begin{figure}[htbp]
   \centering
   \includegraphics[height=1.5in, width=6.2in]{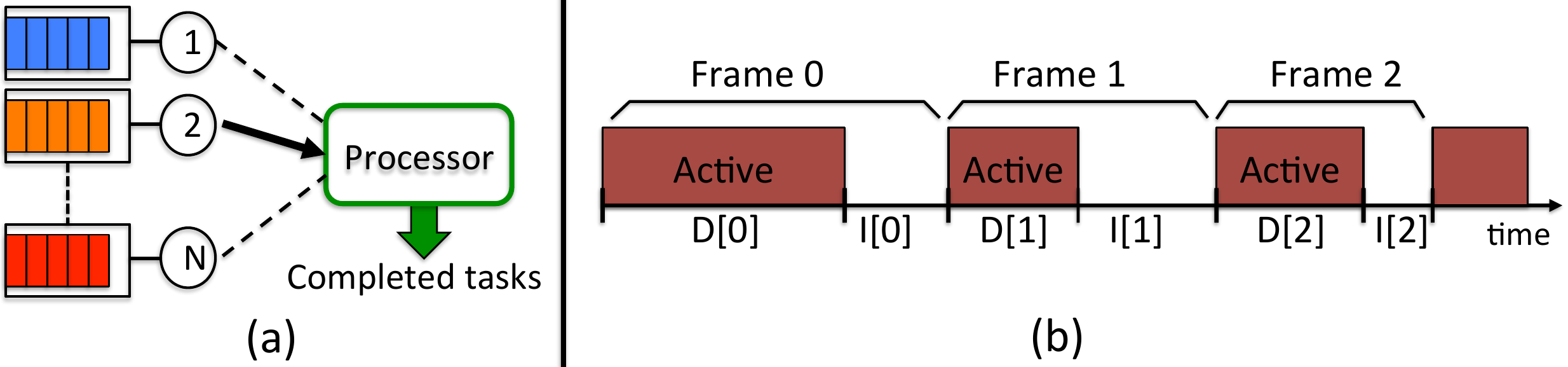} % requires the graphicx package
   \caption{(a) A processor that chooses from one of $N$ classes on each frame $k$. (b) A
   timeline illustrating the active and idle periods for each frame.}
   \label{fig:timeline}
\end{figure}

To illustrate the method, this section considers a particular system.  
%Variations on this system
%can easily be treated using the general methodology described in later sections. 
Consider 
a computer system that repeatedly processes tasks.  There are $N$ classes of tasks, 
where $N$ is a positive integer.   For simplicity, assume each class always has a new task ready
to be performed (this is extended to randomly arriving tasks in 
Section \ref{section:random-task-arrivals}). 
The system operates over time intervals called \emph{frames}.  Each frame $k \in \{0, 1, 2, \ldots\}$ 
begins with an \emph{active period} of size $D[k]$ 
and ends with an \emph{idle period} of size $I[k]$ (see Fig. \ref{fig:timeline}). 
At the beginning of each active period $k$, the controller selects a new task of class $c[k] \in \{1, \ldots, N\}$. 
It also chooses a \emph{processing mode} $m[k]$ from a finite set $\script{M}$ of possible processing
options. These control decisions affect the duration $D[k]$ of the active period and the 
energy $e[k]$ that is incurred. The controller then selects an amount of time $I[k]$ to be idle, where
$I[k]$ is chosen such that $0 \leq I[k] \leq I_{max}$ for some positive number $I_{max}$. 
Choosing $I[k]=0$ effectively skips the idle period, so there can be back-to-back active periods. 
For simplicity, this section assumes that no energy is expended in the idle state. 

Assume that $D[k]$ and $e[k]$
are \emph{random functions} of the class and mode decisions for frame $k$.  
Specifically, assume $D[k]$ and $e[k]$ are conditionally
independent of the past, given the current $(c[k], m[k])$ that is used, with mean values given by 
functions $\hat{D}(c, m)$ and $\hat{e}(c,m)$ defined over $(c,m) \in \{1, \ldots, N\} \times \script{M}$: 
\begin{eqnarray*}
\hat{D}(c,m) \defequiv \expect{D[k]|(c[k], m[k]) = (c,m)} \: \: \: , \: \: \: 
\hat{e}(c,m) \defequiv \expect{e[k]|(c[k], m[k]) = (c,m)} 
\end{eqnarray*}
where the notation ``$a \defequiv b$'' represents ``$a$ is defined to be equal to $b$.''  This paper 
uses $\hat{D}(c[k], m[k])$ and $\hat{e}(c[k], m[k])$ to denote expectations given a 
particular decision $(c[k], m[k])$ for frame $k$: 
\[ \hat{D}(c[k], m[k]) = \expect{D[k]|c[k], m[k]} \: \: \: , \: \: \: \hat{e}(c[k],m[k]) = \expect{e[k]|c[k], m[k]} \]
It is assumed there is a positive value $D_{min}$ such that $D[k] \geq D_{min}$ for all frames $k$, regardless
of the $(c[k],m[k])$ decisions. Thus, 
all frame sizes are at least $D_{min}$ units of time. Further, for technical reasons, 
it is assumed that second moments of $D[k]$ and $e[k]$ are bounded by a finite constant
$\sigma^2$, so that:  
\begin{equation} \label{eq:smb} 
\expect{D[k]^2} \leq \sigma^2 \: \: \: , \: \: \: \expect{e[k]^2} \leq \sigma^2 
\end{equation} 
where \eqref{eq:smb} holds regardless of the policy for selecting $(c[k], m[k])$.
The conditional 
joint distribution of $(D[k], e[k])$, given $(c[k], m[k])$,  is otherwise arbitrary, and 
only the mean values $\hat{D}(c,m)$ and $\hat{e}(c,m)$ are known for each 
$c \in \{1, \ldots, N\}$ and $m \in \script{M}$. 
In the special case of a deterministic system, the functions
$\hat{D}(c,m)$ and $\hat{e}(c,m)$ can be viewed as deterministic mappings from a given control 
action $(c[k],m[k])=(c,m)$ to the actual delay  $D[k]=\hat{D}(c,m)$ and
energy  $e[k]=\hat{e}(c,m)$ experienced on frame $k$, rather than as expectations
of these values. 

\subsection{Examples of Energy-Aware Processing} 

%Consider an example where an embedded microchip performs
%computation for different tasks. 
%Suppose the system uses a multi-processor that 
%can select between one of multiple processing modes for each task.  
Consider an example where a computer system performs computation for different 
tasks. Suppose the system uses a chip multi-processor that can select between one of multiple
processing modes for each task.
For example, this might be done 
using voltage/frequency scaling, or by using a choice of different 
processing cores \cite{low-power-chips}\cite{low-power-chips-murali}. 
Let $\script{M}$ represent the set of processing modes.  For each mode $m \in \script{M}$, 
define: 
\begin{eqnarray*} 
T_{setup}(m) &\defequiv& \mbox{ Setup time for mode $m$.} \\
e_{setup}(m) &\defequiv& \mbox{ Setup energy for mode $m$.} \\
DPI(m) &\defequiv& \mbox{ Average delay-per-instruction for mode $m$.} \\
EPI(m) &\defequiv& \mbox{ Average energy-per-instruction for mode $m$.} 
\end{eqnarray*}
Further suppose that tasks of class $c \in \{1, \ldots, N\}$ have an average number of instructions
equal to $\overline{S}_c$.  For simplicity, suppose the number of instructions in a task is 
independent of the energy and delay of each individual instruction.  
Then the average energy and delay functions $\hat{e}(c,m)$ and $\hat{D}(c,m)$
are: 
\begin{eqnarray*}
\hat{e}(c,m) &=& e_{setup}(m) + \overline{S}_cEPI(m) \\
\hat{D}(c,m) &=& T_{setup}(m) + \overline{S}_cDPI(m)
\end{eqnarray*}
In cases when the system cannot be modeled using $DPI(m)$ and 
$EPI(m)$ functions, the expectations $\hat{e}(c,m)$ and $\hat{D}(c,m)$ can be estimated as
empirical averages of energy and delay observed when processing type $c$ tasks with mode $m$.
While this example assumes all classes have the same set of processing mode 
options $\script{M}$, this can easily be extended to restrict each class $c$ to its own subset
of options $\script{M}_c$.
% (see also the general formulation in Section \ref{section:general-attributes}). 
% This is useful, for example, when some types of 
%tasks can be parallelized and processed simultaneously on different cores of the chip, 
%while others cannot.  

As another example, consider the problem of \emph{wireless
data transmission}.  Here, each task represents a packet of data that must be transmitted.  
Let $\script{M}$ represent the set of wireless transmission options (such as
modulation and coding strategies).  For each $m \in \script{M}$, define $\mu(m)$ as the transmission 
rate (in bits per unit time) under option $m$, and let $power(m)$ be the power used.  For simplicity, 
assume there are no transmission errors. 
Let $\overline{B}_c$ represent the average
packet size for class $c \in \{1, \ldots, N\}$, in units of bits.  Thus:
\begin{eqnarray*}
\hat{e}(c,m) &=& power(m)\overline{B}_c/\mu(m) \\
\hat{D}(c,m) &=& \overline{B}_c/\mu(m) 
\end{eqnarray*}
%In the case of a fading channel with a known probability distribution but where
%fading states cannot be measured at the transmitter, the value $\mu(m)$
%can be viewed as the \emph{average} transmission rate under option $m$, 
%which includes retransmissions due to errors. 
In the case when each transmission mode $m \in \script{M}$ has a known error probability, the functions
$\hat{e}(c,m)$ and $\hat{D}(c,m)$ can be redefined to account for retransmissions.
An important alternative scenario is when channel states are time-varying but can be
measured at the beginning of each frame. This can be treated using the extended theory 
in Section \ref{section:random-events}.

\subsection{Time Averages as Ratios of Frame Averages} \label{section:time-average-as-frame-average} 

The goal is to design a control policy that makes decisions over frames to minimize time average power subject to 
processing each class $n \in \{1, \ldots, N\}$ with rate at least $\lambda_n$, for some desired
processing rates $(\lambda_1, \ldots, \lambda_N)$ that are given.  
Before formalizing this as a mathematical 
optimization, this subsection shows how to write time averages in terms of frame averages. 
Suppose there is a control policy that yields a sequence of energies $\{e[0], e[1], e[2], \ldots\}$
and corresponding frame sizes $\{D[0]+ I[0], D[1] + I[1], D[2] + I[2], \ldots\}$ for each frame $k \in \{0, 1, 2, \ldots\}$. 
The \emph{frame averages} 
$\overline{e}$, $\overline{D}$, $\overline{I}$ are defined: 
\begin{equation} \label{eq:frame-averages} 
 \overline{e} \defequiv \lim_{K\rightarrow\infty}\frac{1}{K}\sum_{k=0}^{K-1} e[k]  \: \: \: , \: \: \: \overline{D} \defequiv \lim_{K\rightarrow\infty}\frac{1}{K}\sum_{k=0}^{K-1} D[k] \: \: \: , \: \: \: \overline{I} \defequiv \lim_{K\rightarrow\infty} \frac{1}{K}\sum_{k=0}^{K-1}I[k]
 \end{equation} 
where, for simplicity, it is assumed  the limits converge to constants with probability 1.
Note that $\overline{e}$ does \emph{not} represent the time average power used by the system, because it
does not consider the amount of time spent in each frame.  The time average power considers the accumulated
energy used divided by the total time, and is written as follows: 
\begin{eqnarray*}
\lim_{K\rightarrow\infty} \frac{\sum_{k=0}^{K-1}e[k]}{\sum_{k=0}^{K-1}(D[k] + I[k])} = \lim_{K\rightarrow\infty}\frac{\frac{1}{K}\sum_{k=0}^{K-1}e[k]}{\frac{1}{K}\sum_{k=0}^{K-1}(D[k]+I[k])} = \frac{\overline{e}}{\overline{D} + \overline{I}} 
\end{eqnarray*}
Therefore, the time average power is equal to the average energy per frame divided by the average frame size. 
This simple observation is often used in \emph{renewal-reward theory} \cite{gallager}\cite{ross-prob}.

For each class $n \in \{1, \ldots, N\}$ and each frame $k$, 
 define an \emph{indicator variable} $1_n[k]$ that is $1$ if the controller chooses to process a class $n$ task on 
 frame $k$, and $0$ else: 
 \[ 1_n[k] \defequiv  \left\{ \begin{array}{ll}
                          1 &\mbox{ if $c[k] = n$} \\
                             0  & \mbox{if $c[k] \neq n$} 
                            \end{array}
                                 \right.\]
 Then $\overline{1}_n \defequiv \lim_{K\rightarrow\infty}\frac{1}{K}\sum_{k=0}^{K-1}1_n[k]$ is the fraction of frames that choose class $n$, and the ratio 
  $\overline{1}_n/(\overline{D} + \overline{I})$ 
 is the time average rate of processing class $n$ tasks, in tasks per unit time. 
 
 The problem of minimizing time average power subject to processing each class $n$ at a rate of 
 at least $\lambda_n$ tasks per unit time  is then mathematically written as follows:
 \begin{eqnarray}
 \mbox{Minimize:} && \frac{\overline{e}}{\overline{D} + \overline{I}} \label{eq:energy} \\
 \mbox{Subject to:} && \frac{\overline{1}_n}{\overline{D} + \overline{I}} \geq \lambda_n \: \: \: \: \forall n \in \{1, \ldots, N\} \label{eq:rate-constraint} \\
 && (c[k], m[k]) \in \{1, \ldots, N\} \times \script{M}   \: \: \: \: \forall k \in \{0, 1, 2, \ldots\} \label{eq:cm-constraint} \\
 && 0 \leq I[k] \leq I_{max}  \: \: \: \: \forall k \in \{0, 1, 2, \ldots\} \label{eq:I-constraint} 
 \end{eqnarray}
 where the objective \eqref{eq:energy} is average power, the constraint \eqref{eq:rate-constraint} ensures the 
 processing rate of each  class $n$ is at least $\lambda_n$, and 
 constraints \eqref{eq:cm-constraint}-\eqref{eq:I-constraint} ensure that $c[k] \in \{1, \ldots, N\}$, 
 $m[k] \in \script{M}$, and $0 \leq I[k] \leq I_{max}$ for each frame $k$. 
 
\subsection{Relation to Frame Average Expectations} 

The problem \eqref{eq:energy}-\eqref{eq:I-constraint} is defined by frame averages.  This subsection 
shows that frame averages are related to frame average \emph{expectations}, and hence can be related
to the expectation functions $\hat{D}(c,m)$, $\hat{e}(c,m)$.  Consider any (possibly randomized)
control algorithm for selecting $(c[k], m[k])$ over frames, and assume this gives rise to well defined
expectations $\expect{e[k]}$ for each frame $k$. By the law of iterated expectations, it follows that 
for any given frame $k \in \{0, 1,2 , \ldots\}$: 
\begin{equation} \label{eq:iterated-expectations} 
\expect{e[k]} = \expect{\: \expect{e[k]|c[k], m[k]}\: } = \expect{\: \hat{e}(c[k], m[k]) \: } 
\end{equation} 
Furthermore, because second moments are bounded by a  constant $\sigma^2$ on each frame $k$, 
the \emph{bounded moment convergence theorem} (given in Appendix A) 
ensures that if the frame average energy converges
to a constant $\overline{e}$ with probability 1, as defined in \eqref{eq:frame-averages}, then $\overline{e}$ is the 
same as the frame average expectation: 
\[ \lim_{K\rightarrow\infty}\frac{1}{K}\sum_{k=0}^{K-1} e[k] = \overline{e} = \lim_{K\rightarrow\infty} \frac{1}{K}\sum_{k=0}^{K-1}\expect{\hat{e}(c[k],m[k])} \] 
The same goes for the quantities $\overline{D}$, $\overline{I}$, $\overline{1}_n$.  Hence, one can interpret
the problem  \eqref{eq:energy}-\eqref{eq:I-constraint} using frame average expectations, rather than pure
frame averages.  

\subsection{An Example with One Task Class} \label{section:one-class-example} 

Consider a simple example with only one type of task.  The system processes a new task of this
type at the beginning of every busy period.  The energy and delay functions can then
be written purely in terms of the 
processing mode $m \in \script{M}$, so that we have $\hat{e}(m)$ and $\hat{D}(m)$. 
Suppose there are only two processing mode options, so that $\script{M} = \{1, 2\}$, and that
each option leads to a  deterministic energy and delay, as given below: 
\begin{eqnarray}
m[k] = 1 &\implies& (e[k], D[k]) = (\hat{e}(1), \hat{D}(1)) =  (1, 7) \label{eq:m-value1} \\
m[k] = 2 &\implies& (e[k], D[k]) = (\hat{e}(2), \hat{D}(2)) =  (3, 4) \label{eq:m-value2} 
\end{eqnarray}
Option $m[k]=1$ requires 1 unit of energy but 7 units of processing time. Option $m[k]=2$ is more energy-expensive (requiring $3$ units of energy) but is faster (taking only 4 time units).   The idle time $I[k]$ is chosen every frame
in the interval $[0, 10]$, so that $I_{max} = 10$. 

\subsubsection{No constraints} 

For this system, suppose we seek to minimize average 
power $\overline{e}/(\overline{D} + \overline{I})$, with no processing rate constraint.  
Consider three possible algorithms:  
\begin{enumerate} 
\item Always use $m[k]=1$ and $I[k] = I$ for all frames $k$, for some constant 
$I \in [0, 10]$. 
\item Always use $m[k] = 2$ and $I[k] =I$ for all frames $k$, for some constant $I \in [0, 10]$. 
\item Always use $I[k] = I$ for all frames $k$, for some constant $I \in [0, 10]$. However, each frame $k$, independently
choose $m[k]=1$ with probability $p$, and $m[k]=2$ with probability $1-p$ (for some $p$ that satisfies
$0 \leq p \leq 1$). 
\end{enumerate} 

Clearly the third algorithm contains the first two for $p=1$ and $p=0$, respectively.  The time average power
under each algorithm is: 
\begin{eqnarray*}
\mbox{$m[k] = 1$ always} &\implies& \frac{\overline{e}}{\overline{D} + \overline{I}} = \frac{\hat{e}(1)}{\hat{D}(1) + I} = \frac{1}{7 + I} \\
\mbox{$m[k]=2$ always} &\implies& \frac{\overline{e}}{\overline{D} + \overline{I}} = \frac{\hat{e}(2)}{\hat{D}(2) + I} =\frac{3}{4+I} \\
\mbox{probabilistic rule} &\implies& \frac{\overline{e}}{\overline{D} + \overline{I}} = \frac{p\hat{e}(1) + (1-p)\hat{e}(2)}{p\hat{D}(1) + (1-p)\hat{D}(2) + I} = \frac{1p + 3(1-p)}{7p + 4(1-p) + I}
\end{eqnarray*}

It is clear that in all three cases, we should choose $I = 10$ to minimize average power. 
Further, it is clear that always choosing $m[k]=1$ is better than always choosing $m[k]=2$.  
However, it is not immediately obvious if a randomized mode selection rule can do even better.  The 
answer is no:  In this case, power is minimized by choosing $m[k] = 1$ and $I[k] = 10$ for all frames $k$, 
yielding average power $1/17$. 

This fact holds more generally: Let
$a(\alpha)$ and $b(\alpha)$ be any deterministic real valued 
functions defined over general \emph{actions} $\alpha$ 
that are chosen in an abstract \emph{action space} $\script{A}$.  Assume the functions are bounded, and that 
there is a value $b_{min}>0$
such that $b(\alpha) \geq b_{min}$ for all $\alpha \in \script{A}$.  Consider designing
a randomized choice of 
$\alpha$ that minimizes $\expect{a(\alpha)}/\expect{b(\alpha)}$, where the expectations are with respect
to the randomness in the $\alpha$ selection.  The next lemma shows this is done 
by  \emph{deterministically} choosing an action $\alpha \in \script{A}$ that 
minimizes $a(\alpha)/b(\alpha)$. 

\begin{lem} \label{lem:deterministic-min} 
Under the assumptions of the preceding paragraph, for any randomized choice of $\alpha\in\script{A}$ 
we have: 
\[ \frac{\expect{a(\alpha)}}{\expect{b(\alpha)}} \geq \inf_{\alpha\in\script{A}} \left[\frac{a(\alpha)}{b(\alpha)}\right] \]
and so the infimum of the ratio over the class of deterministic decisions yields a value that is less than or equal
to that of any randomized selection.
\end{lem}  

\begin{proof}  
Consider any randomized policy that yields expectations $\expect{a(\alpha)}$ and $\expect{b(\alpha)}$. 
Without loss of generality, assume these expectations are  
achieved by a policy that randomizes over a finite set of $M$  
actions $\alpha_1, \ldots, \alpha_M$ in $\script{A}$ with some probabilities $p_1, \ldots, p_M$:\footnote{Indeed, because the set $\script{S} = \{(a(\alpha), b(\alpha)) \mbox{ such that $\alpha \in \script{A}$\}}$ is bounded, the 
expectation $(\expect{a(\alpha)}, \expect{b(\alpha)})$ is finite and is contained in the convex hull of $\script{S}$. 
Thus, $(\expect{a(\alpha)}, \expect{b(\alpha)})$ is a convex combination of a finite number of points in $\script{S}$. %The set $\script{S}$ is 2 dimensional, and so 
%Caratheodory's theorem further 
%ensures we can find a convex combination that uses $M\leq 3$ points.
} 
\[ \expect{a(\alpha)} = \sum_{m=1}^Mp_ma(\alpha_m) \: \: , \: \: \expect{b(\alpha)} = \sum_{m=1}^Mp_mb(\alpha_m) \]
Then, because $b(\alpha_m) > 0$ for all $\alpha_m$, we have: 
\begin{eqnarray*}
\frac{\expect{a(\alpha)}}{\expect{b(\alpha)}} &=& \frac{\sum_{m=1}^Mp_ma(\alpha_m)}{\sum_{m=1}^Mp_mb(\alpha_m)} \\
&=& \frac{\sum_{m=1}^Mp_mb(\alpha_m)[a(\alpha_m)/b(\alpha_m)]}{\sum_{m=1}^Mp_mb(\alpha_m)} \\
&\geq& \frac{\sum_{m=1}^Mp_mb(\alpha_m)\inf_{\alpha\in\script{A}}[a(\alpha)/b(\alpha)]}{\sum_{m=1}^Mp_mb(\alpha_m)} \\
&=&   \inf_{\alpha\in\script{A}} [a(\alpha)/b(\alpha)]
\end{eqnarray*}
\end{proof}  

%Lemma \ref{lem:deterministic-min} is important because we soon develop a technique that
%solves constrained problems of the type \eqref{eq:energy}-\eqref{eq:I-constraint} by making a sequence
%of unconstrained decisions every frame. The unconstrained decisions involve minimizing a ratio 
%of the type considered in the lemma. 

\subsubsection{One Constraint}  

The preceding subsection shows that unconstrained problems can be solved by 
deterministic actions. 
This is not true for constrained problems. This subsection shows that adding just a single constraint 
often \emph{necessitates} the use of 
randomized actions.   Consider the same problem as before, with two 
choices for $m[k]$, and with the 
same $\hat{e}(m)$ and $\hat{D}(m)$ values given in \eqref{eq:m-value1}-\eqref{eq:m-value2}. 
We want to minimize $\overline{e}/(\overline{D} + \overline{I})$ subject to 
$1/(\overline{D} + \overline{I}) \geq 1/5$, where $1/(\overline{D} + \overline{I})$ is the rate of 
processing jobs.  The constraint is equivalent to $\overline{D} + \overline{I} \leq 5$. 

Assume the algorithm chooses $I[k]$ over frames to yield an average $\overline{I}$ that is somewhere
in the interval $0 \leq \overline{I} \leq 10$.  Now consider different algorithms for selecting $m[k]$.  
If we choose $m[k]=1$ always (so that $(\hat{e}(1), \hat{D}(1)) = (1, 7)$), then: 
\begin{eqnarray*}
\mbox{$m[k]=1$ always} &\implies& \frac{\overline{e}}{\overline{D}  + \overline{I}} = \frac{1}{7 + \overline{I}} \: \: \: \:  , \: \: \: \: 
 \overline{D} + \overline{I} = 7 + \overline{I}
\end{eqnarray*}
and so it is impossible to meet the 
constraint $\overline{D} + \overline{I} \leq 5$, because $7 + \overline{I} > 5$.  

If we choose  $m[k]=2$ always (so that $(\hat{e}(2), \hat{D}(2)) = (3, 4)$), then:  
\begin{eqnarray*}
\mbox{$m[k]=2$ always} &\implies& \frac{\overline{e}}{\overline{D}  + \overline{I}} = \frac{3}{4 + \overline{I}} \: \: \: \: , \: \: \: \: 
 \overline{D} + \overline{I} = 4 + \overline{I}
\end{eqnarray*}
It is clear that we can meet the constraint by choosing $\overline{I}$ so that $0 \leq \overline{I} \leq 1$, 
and power is minimized in this setting 
by using $\overline{I} = 1$.  This can be achieved, for example, by using $I[k] =1$
for all frames $k$.  This meets the processing rate constraint with equality:  $\overline{D} + \overline{I} =  4 + 1 = 5$. 
Further, it yields average power $\overline{e}/(\overline{D} + \overline{I}) = 3/5 = 0.6$. 

However, it is possible to reduce average power while also meeting the constraint with equality by
using the following randomized policy (which can be shown to be optimal):  Choose $I[k]=0$ for all frames 
$k$, so that $\overline{I} = 0$.  Then every frame $k$, independently choose $m[k]=1$ with probability 1/3, 
and $m[k]=2$ with probability 2/3.  We then have: 
\[ \overline{D} + \overline{I} = (1/3)7 + (2/3)4 + 0 = 5 \]
and so the processing rate constraint is met with equality.  However, average power is: 
\[ \frac{\overline{e}}{\overline{D} + \overline{I}} = \frac{(1/3)1 + (2/3)3}{(1/3)7 + (2/3)4 + 0} = 7/15 \approx 0.466667 \]
This is a significant savings over the average power of $0.6$ from the  deterministic policy.

\subsection{The Linear Fractional Program for Task Scheduling} 

Now consider the general problem \eqref{eq:energy}-\eqref{eq:I-constraint} for minimizing average 
power subject to average processing rate constraints for each of the $L$ classes.  Assume the 
problem is \emph{feasible}, so that it is possible to choose actions $(c[k], m[k], I[k])$ over frames to 
meet the desired constraints \eqref{eq:rate-constraint}-\eqref{eq:I-constraint}.  It can be shown that
an optimal solution can be achieved over the class of \emph{stationary and randomized policies} with 
the following structure:  Every frame, independently choose vector $(c[k], m[k])$ 
with some probabilities $p(c,m)= Pr[(c[k], m[k]) = (c,m)]$.  Further, use 
a constant idle time $I[k] = I$ for all frames $k$, for some constant $I$ that satisfies $0 \leq I \leq I_{max}$. 
Thus, the problem can be written as the following \emph{linear fractional program} with unknowns
$p(c,m)$ and $I$ and known constants $\hat{e}(c,m)$, $\hat{D}(c,m)$, $\lambda_n$, and $I_{max}$: 
\begin{eqnarray}
\mbox{Minimize:} && \frac{\sum_{c=1}^N\sum_{m\in\script{M}}p(c,m) \hat{e}(c,m)}{I + \sum_{c=1}^N\sum_{m\in\script{M}}p(c,m)\hat{D}(c,m)} \label{eq:lf1} \\
\mbox{Subject to:} && \frac{\sum_{m\in\script{M}}p(n,m)}{I + \sum_{c=1}^N\sum_{m\in\script{M}}p(c,m)\hat{D}(c,m)} \geq \lambda_n \: \: \: \:  \forall n \in \{1, \ldots, N\} \label{eq:lf2} \\
&& 0 \leq I \leq I_{max} \label{eq:lf3} \\
&& p(c,m) \geq 0 \: \: \: \: \forall c \in \{1, \ldots, N\}, \: \: \forall m \in \script{M} \label{eq:lf4}  \\
&& \sum_{c=1}^N\sum_{m\in\script{M}} p(c,m) = 1 \label{eq:lf5}  
\end{eqnarray}
where the numerator and denominator in \eqref{eq:lf1} are equal to $\overline{e}$ and $\overline{D} + \overline{I}$,
respectively, under this randomized algorithm, the numerator in the left-hand-side of \eqref{eq:lf2} is equal 
to $\overline{1}_n$, and the constraints \eqref{eq:lf4}-\eqref{eq:lf5} specify that $p(c,m)$ must be a valid
probability mass function. 

Linear fractional programs  can be solved in several ways.  One method uses a 
nonlinear change of variables to map the problem to a convex program \cite{boyd-convex}.  However, 
this method does not admit an \emph{online} implementation, because time averages are not preserved
through the nonlinear change of variables.   Below, an online algorithm is presented that makes decisions
every frame $k$.  The algorithm is \emph{not} a stationary and randomized algorithm as described above.  
However, it
yields time averages that satisfy the desired constraints of the 
problem \eqref{eq:energy}-\eqref{eq:I-constraint}, with a time average power expenditure that can be
pushed arbitrarily close to the optimal value.  A significant advantage of this approach is that it extends 
to treat cases with random task arrivals, without requiring knowledge of the $(\lambda_1, \ldots, \lambda_N)$
arrival rates, and to treat other problems with observed random events, without requiring knowledge
of the probability distribution for these events.  These extensions are shown in later sections. 

For later analysis, it is useful to write \eqref{eq:lf1}-\eqref{eq:lf5} in a simpler form. 
Let $power^{opt}$ be the optimal time average power for the above linear fractional program, 
achieved by some probability distribution $p^*(c,m)$ and idle time $I^*$ that satisfies $0 \leq I^* \leq I_{max}$. 
Let $(c^*[k], m^*[k], I^*[k])$ represent the frame $k$ decisions under this 
stationary and randomized policy.  Thus: 
\begin{eqnarray} 
\frac{\expect{\hat{e}(c^*[k], m^*[k])}}{\expect{I^*[k] + \hat{D}(c^*[k], m^*[k])}} &=& power^{opt}  \label{eq:sr1} \\
\frac{\expect{1_n^*[k]}}{\expect{I^*[k] + \hat{D}(c^*[k], m^*[k])}} &\geq& \lambda_n \: \: \: \: \forall n \in \{1, \ldots, N\} \label{eq:sr2}  
\end{eqnarray}  
where $1_n^*[k]$ is an indicator function that is $1$ if $c^*[k]=n$, and $0$ else.  The numerator
and denominator of \eqref{eq:sr1} correspond to those of \eqref{eq:lf1}. Likewise, the constraint 
\eqref{eq:sr2} corresponds to \eqref{eq:lf2}. 

\subsection{Virtual Queues} 

\begin{figure}[htbp]
   \centering
   \includegraphics[height=.7in, width=4in]{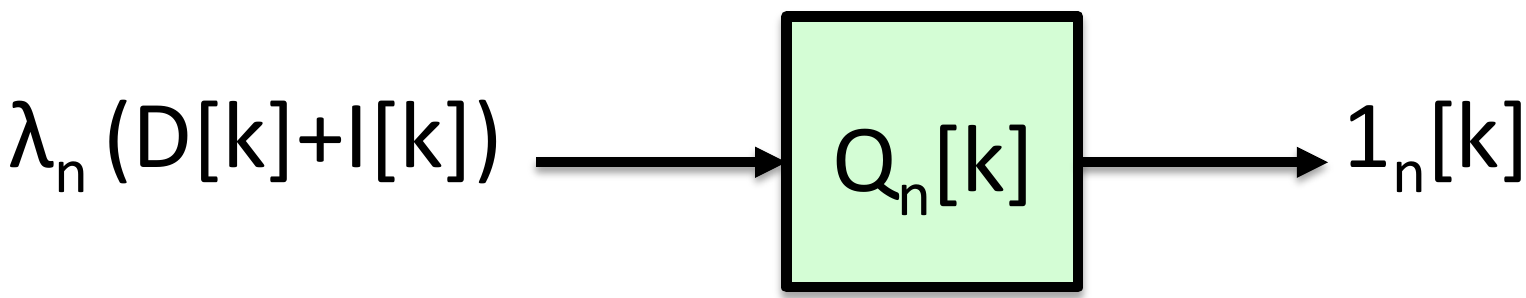} % requires the graphicx package
   \caption{An illustration of the virtual queue $Q_n[k]$ from equation \eqref{eq:z-update1}.}
   \label{fig:virtual-queue}
\end{figure}

To solve the problem \eqref{eq:energy}-\eqref{eq:I-constraint}, we first consider the 
constraints \eqref{eq:rate-constraint}, which are equivalent to the constraints: 
\begin{equation} 
 \lambda_n(\overline{D} + \overline{I}) \leq \overline{1}_n \: \: \: \: \forall n \in \{1, \ldots, N\}  \label{eq:c1} 
 \end{equation} 
For each constraint $n \in \{1, \ldots, N\}$, define a \emph{virtual queue} $Q_n[k]$ that is updated
on frames $k \in \{0, 1, 2, \ldots\}$ by: 
\begin{equation} \label{eq:z-update1} 
Q_n[k+1] = \max[Q_n[k] + \lambda_n(D[k] + I[k]) - 1_n[k], 0] 
\end{equation} 
The initial condition $Q_n[0]$ can be any non-negative value. For simplicity, it is assumed
throughout that $Q_n[0]=0$ for all $n \in \{1, \ldots, N\}$. The update \eqref{eq:z-update1} can be viewed
as a discrete time queueing equation, where $Q_n[k]$ is the backlog on frame $k$, $\lambda_n(D[k] + I[k])$
is an effective amount of ``new arrivals,'' and $1_n[k]$ is the amount of 
``offered service''  (see Fig. \ref{fig:virtual-queue}).  The intuition is that if all virtual queues $Q_n[k]$
are \emph{stable}, then the average ``arrival rate'' $\lambda_n(\overline{D} + \overline{I})$ must be less than or 
equal to the average ``service rate'' $\overline{1}_n$, which ensures the desired constraint \eqref{eq:c1}. 
This is made precise in the following lemma. 

\begin{lem} \label{lem:virtual-queues} (Virtual Queues) Suppose $Q_n[k]$ has update equation 
given by \eqref{eq:z-update1}, with any non-negative initial condition.   

a) For all $K \in \{1, 2, 3, \ldots\}$ we have: 
\begin{equation} \label{eq:parta} 
 \frac{1}{K}\sum_{k=0}^{K-1}[\lambda_n(D[k] + I[k]) - 1_n[k]] \leq \frac{Q_n[K] - Q_n[0]}{K} 
 \end{equation} 

b) If $\lim_{K\rightarrow\infty} Q_n[K]/K = 0$ with probability 1,  then: 
\begin{equation} \label{eq:desired1} 
\limsup_{K\rightarrow\infty} \frac{1}{K}\sum_{k=0}^{K-1}[\lambda_n(D[k] + I[k]) - 1_n[k]] \leq 0 \: \: \: \mbox{ with probability 1} 
\end{equation} 

c) If $\lim_{K\rightarrow\infty} \expect{Q_n[K]}/K = 0$, then: 
\begin{equation} \label{eq:desired2} 
 \limsup_{K\rightarrow\infty} [\lambda_n(\overline{D}[K] + \overline{I}[K]) - \overline{1}_n[K]] \leq 0 
 \end{equation} 
where $\overline{D}[K]$, $\overline{I}[K]$, $\overline{1}_n[K]$ are defined: 
\[ \mbox{$\overline{D}[K] \defequiv \frac{1}{K}\sum_{k=0}^{K-1}\expect{D[k]} \: \: , \: \: \overline{I}[K] \defequiv \frac{1}{K}\sum_{k=0}^{K-1}\expect{I[k]} \: \: , \: \: \overline{1}_n[K] \defequiv \frac{1}{K}\sum_{k=0}^{K-1}\expect{1_n[k]}$} \]
\end{lem} 

\begin{proof} From \eqref{eq:z-update1} we have for all $k \in \{0, 1, 2, \ldots\}$: 
\[ Q_n[k+1] \geq Q_n[k] + \lambda_n(D[k] + I[k]) - 1_n[k] \]
Fixing a positive integer $K$ and summing the above over $k \in \{0, \ldots, K-1\}$ yields: 
\[ Q_n[K] - Q_n[0] \geq \sum_{k=0}^{K-1} [\lambda_n(D[k] + I[k]) - 1_n[k]] \]
Dividing the above by $K$ proves part (a).  Part (b) follows from  \eqref{eq:parta} 
by taking a $\limsup$. Part (c) follows by first taking expectations of \eqref{eq:parta} and then taking a $\limsup$. 
\end{proof} 

Inequality \eqref{eq:parta} shows that the value $Q_n[K]/K$ bounds the amount by which the 
desired constraint for class $n$ is violated by the time averages achieved over the first $K$ frames.  
Suppose that $D[k]$, $I[k]$, and $1_n[k]$ have frame averages
that converge to constants $\overline{D}$, $\overline{I}$, $\overline{1}_n$ with probability 1.  
Part (c) of Lemma \ref{lem:virtual-queues} 
indicates that if 
$\lim_{K\rightarrow\infty} \expect{Q_n[K]}/K = 0$ for all $n \in \{1, \ldots, N\}$,  then $\lambda_n(\overline{D} + \overline{I}) \leq \overline{1}_n$ for all $n \in \{1, \ldots, N\}$. 

In the language of queueing theory, a discrete time queue $Q[k]$ is said to be \emph{rate stable} 
if   $\lim_{k\rightarrow\infty} Q[k]/k = 0$ with probability 1, and is \emph{mean rate stable} if 
$\lim_{k\rightarrow\infty} \expect{Q[k]}/k = 0$ \cite{sno-text}.  With this terminology, the
above lemma shows that if $Q_n[k]$ is 
rate stable then 
\eqref{eq:desired1} holds, and if $Q_n[k]$ is mean rate stable then \eqref{eq:desired2} holds. 

\subsection{The Drift-Plus-Penalty Ratio} \label{section:dppr} 

To stabilize the queues while minimizing time average power, we use \emph{Lyapunov
optimization theory}, which gives rise to 
the \emph{drift-plus-penalty ratio algorithm} \cite{sno-text}. 
First define $L[k]$ as the sum of the squares of all queues on frame $k$ (divided by $2$ for convenience
later): 
\[ L[k] \defequiv \frac{1}{2}\sum_{n=1}^NQ_n[k]^2 \]
$L[k]$ is often called a \emph{Lyapunov function}, and acts as a scalar measure of the size of the 
queues.  Intuitively, keeping $L[k]$ small leads to stable queues, and we should take actions that
tend to shrink $L[k]$ from one frame to the next.  Define $\Delta[k]$ as the \emph{Lyapunov drift}, being the 
difference in the Lyapunov function from one frame to the next: 
\[ \Delta[k] \defequiv L[k+1] - L[k] \]

Taking actions to minimize $\Delta[k]$ every frame can be shown to ensure the desired 
constraints are satisfied whenever it is possible to satisfy them, but does not incorporate power minimization. 
To incorporate this, every frame $k$ we observe the current queue vector $\bv{Q}[k] = (Q_1[k], \ldots, Q_N[k])$
and choose control actions $(c[k], m[k], I[k])$ 
to minimize a bound on the following \emph{drift-plus-penalty ratio}: 
\[ \frac{\expect{ \Delta[k] + Ve[k]| \bv{Q}[k]}}{\expect{D[k] + I[k] | \bv{Q}[k]} } \]
where $V$ is a non-negative parameter that weights the extent to which power minimization is emphasized. 
The intuition is that the numerator incorporates both drift and energy. The denominator ``normalizes'' this
by the expected frame size, with the understanding that average power must include both energy and
frame size.  We soon show that this intuition is correct, in that all 
desired time average constraints are satisfied, and the time average power is within $O(1/V)$ of
the optimal value $power^{opt}$.  Hence, average power 
can be pushed arbitrarily close to optimal by using a sufficiently
large value of $V$.  The tradeoff is that average queue sizes grow with $V$, which impacts the convergence
time required to satisfy the desired constraints.  

The drift-plus-penalty ratio method was first developed for the context of restless bandit systems in 
\cite{chihping-utility-round-robin}\cite{chihping-utility-round-robin-wiopt}.  The method was used for
optimization of renewal systems in \cite{sno-text}\cite{renewals-asilomar2010}, which treat problems
similar to those considered in this paper. In the special case when all frame sizes are fixed and 
equal to one unit of time (a \emph{time slot}), and when $V=0$, the method reduces to observing queues $\bv{Q}[k]$ 
every slot $k$ and taking actions to minimize a bound on $\expect{\Delta[k]|\bv{Q}[k]}$. This is the 
rule that generates the classic max-weight scheduling 
algorithms for queue stability (without performance optimization), 
developed by Tassiulas and Ephremides in \cite{tass-radio-nets}\cite{tass-server-allocation}. 
For systems with unit time slots but with $V>0$, the \emph{drift-plus-penalty ratio} technique 
reduces to the \emph{drift-plus-penalty} technique of \cite{neely-thesis}\cite{now}\cite{neely-energy-it}, 
which treats joint queue stability and penalty minimization in systems with unit size slots.

\subsubsection{Bounding the Drift-Plus-Penalty Ratio}

To construct an explicit algorithm, we first bound the drift-plus-penalty ratio. 

\begin{lem} \label{lem:dpp-bound1} For all frames $k\in\{0, 1, 2, \ldots\}$, all possible 
$\bv{Q}[k]$, and under any decisions for $(c[k], m[k], I[k])$, we have: 
\begin{eqnarray}
\hspace{-.3in}\frac{\expect{\Delta[k] + Ve[k]|\bv{Q}[k]}}{\expect{D[k] + I[k]|\bv{Q}[k]}} &\leq& \frac{B}{\expect{D[k]+I[k]|\bv{Q}[k]}}
+ \frac{\expect{V\hat{e}(c[k],m[k])|\bv{Q}[k]}}{\expect{\hat{D}(c[k], m[k]) + I[k]|\bv{Q}[k]}} \nonumber \\
&& + \frac{\sum_{n=1}^NQ_n[k]\expect{\lambda_n(\hat{D}(c[k],m[k]) + I[k]) - 1_n[k]|\bv{Q}[k]}}{\expect{\hat{D}(c[k], m[k]) + I[k]|\bv{Q}[k]}} \label{eq:dpp1} 
\end{eqnarray}
where $B$ is a constant that satisfies the following for all possible $\bv{Q}[k]$ and all policies: 
\[ B \geq \frac{1}{2}\sum_{n=1}^N\expect{(\lambda_n(D[k] + I[k]) - 1_n[k])^2|\bv{Q}[k]} \]
Such a constant $B$ exists by the second moment boundedness assumptions \eqref{eq:smb}. 
\end{lem} 
\begin{proof} 
Note by iterated expectations that (similar to \eqref{eq:iterated-expectations}):\footnote{In more detail, by iterated 
expectations we have 
$\expect{D[k]|\bv{Q}[k]} = \expect{\expect{D[k]|(c[k], m[k]), \bv{Q}[k]}|\bv{Q}[k]}$, and $\expect{D[k]|(c[k], m[k]), \bv{Q}[k]} = \expect{D[k]|(c[k], m[k])}$ because $D[k]$ is conditionally independent of the past given the current $(c[k], m[k])$ used.} 
\[ \expect{D[k]|\bv{Q}[k]} = \expect{\hat{D}(c[k], m[k])|\bv{Q}[k]} \: \: , \: \: \expect{e[k]|\bv{Q}[k]} = \expect{\hat{e}(c[k],m[k])|\bv{Q}[k]} \]
Thus, the denominator is common for all terms of inequality \eqref{eq:dpp1}, and it suffices to prove: 
\begin{equation} \label{eq:suffices} 
\expect{\Delta[k] | \bv{Q}[k]} \leq B + \sum_{n=1}^NQ_n[k]\expect{\lambda_n(\hat{D}(c[k],m[k]) + I[k]) - 1_n[k]|\bv{Q}[k]}
\end{equation} 
To this end, by squaring \eqref{eq:z-update1} and noting that $\max[x,0]^2 \leq x^2$, 
we have for each $n$: 
\begin{eqnarray*}
\frac{1}{2}Q_n[k+1]^2 &\leq& \frac{1}{2}(Q_n[k] + \lambda_n (D[k] + I[k]) - 1_n[k])^2 \\
&=& \frac{1}{2}Q_n[k]^2 + \frac{1}{2}(\lambda_n(D[k]+I[k]) - 1_n[k])^2 + Q_n[k](\lambda_n(D[k]+I[k]) - 1_n[k])
\end{eqnarray*}
Summing the above over $n \in \{1, \ldots, N\}$ and using the definition of $\Delta[k]$ gives: 
\[ \Delta[k] \leq \frac{1}{2}\sum_{n=1}^N(\lambda_n(D[k] + I[k]) - 1_n[k])^2 + \sum_{n=1}^NQ_n[k](\lambda_n(D[k]+ I[k]) - 1_n[k]) \]
Taking conditional expectations given $\bv{Q}[k]$ and using the bound $B$ 
proves \eqref{eq:suffices}.
\end{proof} 

\subsubsection{The Task Scheduling Algorithm} 

Our algorithm takes actions every frame to minimize the last two terms on the right-hand-side of 
the drift-plus-penalty ratio bound \eqref{eq:dpp1}.  The only part of these terms 
that we have control over on frame $k$ (given the observed $\bv{Q}[k]$) is given below: 
\[ \frac{\expect{V\hat{e}(c[k], m[k]) - \sum_{n=1}^NQ_n[k]1_n[k]|\bv{Q}[k]}}{\expect{\hat{D}(c[k], m[k]) + I[k]|\bv{Q}[k]}} \]
Recall from Lemma \ref{lem:deterministic-min} 
that minimizing the above ratio of expectations is accomplished
over a deterministic choice of $(c[k], m[k], I[k])$. Thus, every frame $k$ we perform the 
following:  

\begin{itemize} 
\item Observe queues $\bv{Q}[k] = (Q_1[k], \ldots, Q_N[k])$.  Then choose $c[k] \in \{1, \ldots, N\}$, 
$m[k] \in \script{M}$, and $I[k]$ such that $0 \leq I[k] \leq I_{max}$ to minimize: 
\begin{eqnarray}
\frac{V\hat{e}(c[k], m[k]) - Q_{c[k]}[k]}{\hat{D}(c[k],m[k]) + I[k]} \label{eq:minimize} 
\end{eqnarray}
\item Update queues $Q_n[k]$ for each $n \in \{1, \ldots, N\}$ 
via \eqref{eq:z-update1}, using the $D[k]$, $I[k]$, and $1_n[k]$ values
that result from the decisions $c[k]$, $m[k]$, $I[k]$ that minimized \eqref{eq:minimize}  
\end{itemize} 

\subsubsection{Steps to minimize \eqref{eq:minimize}} 

Here we elaborate on how to perform the minimization in  \eqref{eq:minimize} for each frame $k$. 
For each $c \in \{1, \ldots, N\}$ and $m \in \script{M}$, 
 define  $idle(c,m)$ as the value of $I[k]$ that minimizes
\eqref{eq:minimize}, given that we have $(c[k], m[k]) = (c,m)$. It is easy to see that: 
\[ idle(c,m) =  \left\{ \begin{array}{ll}
                          0 &\mbox{ if $V\hat{e}(c,m) - Q_c[k] \leq 0$} \\
                            I_{max}  & \mbox{ otherwise} 
                            \end{array}
                                 \right.\]
Now define $val(c,m)$ by: 
\[ val(c,m) = \frac{V\hat{e}(c,m) - Q_c[k]}{\hat{D}(c,m) + idle(c,m)} \]
Then we choose $(c[k], m[k])$ as the minimizer of $val(c,m)$ over $c \in \{1, \ldots, N\}$ and $m \in \script{M}$, 
breaking ties arbitrarily, and choose $I[k] = idle(c[k], m[k])$. 
Note that this algorithm chooses $I[k] = 0$ or $I[k] = I_{max}$ on every frame $k$.  
Nevertheless, it results in a \emph{frame average} $\overline{I}$ that approaches optimality for large $V$. 

\subsection{Performance of the Task Scheduling Algorithm} 

For simplicity, the performance theorem is presented in terms of zero initial conditions.
It is assumed throughout that the problem \eqref{eq:energy}-\eqref{eq:I-constraint} is feasible,
so that it is possible to satisfy the constraints. 

\begin{thm} \label{thm:perf1} Suppose $Q_n[0] = 0$ for all $n \in \{1, \ldots, N\}$, and that the problem 
\eqref{eq:energy}-\eqref{eq:I-constraint} is feasible. Then under the above task scheduling
algorithm: 

a) For all frames $K \in \{1, 2, 3, \dots\}$ we have:\footnote{The right-hand-side in 
\eqref{eq:g1} can be simplified to $power^{opt} + B/(VD_{min})$, since all frames are at least
$D_{min}$ in size.} 
\begin{equation} \label{eq:g1} 
 \frac{\overline{e}[K]}{\overline{D}[K] + \overline{I}[K]} \leq power^{opt} + \frac{B}{V(\overline{D}[K]+\overline{I}[K])} 
 \end{equation} 
 where $B$ is defined in Lemma \ref{lem:dpp-bound1}, 
 $power^{opt}$ is the minimum power solution for the problem \eqref{eq:energy}-\eqref{eq:I-constraint},
and $\overline{e}[K]$, $\overline{D}[K]$, $\overline{I}[K]$ are defined by: 
\[ \mbox{$\overline{e}[K] \defequiv \frac{1}{K}\sum_{k=0}^{K-1} \expect{e[k]} \: \: , \: \: \overline{D}[K] \defequiv \frac{1}{K}\sum_{k=0}^{K-1} \expect{D[k]} \: \: , \: \: \overline{I}[K] \defequiv \frac{1}{K}\sum_{k=0}^{K-1} \expect{I[k]}$} \]

b) The desired constraints 
\eqref{eq:desired1} and \eqref{eq:desired2} are satisfied for all $n \in \{1, \ldots, N\}$.  
Further, we have for each frame $K \in \{1, 2, 3, \ldots\}$: 
\begin{equation} \label{eq:g2} 
 \frac{\expect{\norm{\bv{Q}[k}}}{K} \leq \sqrt{\frac{2(B+V\beta)}{K}} 
 \end{equation} 
where $\norm{\bv{Q}[k]} \defequiv \sqrt{\sum_{n=1}^NQ_n[k]^2}$ is the norm of the queue vector (being
at least as large as each component $Q_n[k]$), 
and $\beta$ is a constant that satisfies the following for all frames $k$: 
\[ \beta \geq \expect{power^{opt}(D[k] + I[k]) - e[k]} \]
Such a 
constant $\beta$ exists because the second moments (and hence first moments) are bounded. 
\end{thm} 

In the special case of a deterministic system, all expectations of the above theorem 
can be removed, and the results hold deterministically for all frames $K$.  Theorem \ref{thm:perf1} 
indicates that average power can be pushed arbitrarily close to $power^{opt}$, using the $V$ parameter
that affects an $O(1/V)$ performance gap given in \eqref{eq:g1}.  The tradeoff is that $V$ increases the 
expected size of $\expect{Q_n[K]}/K$ as shown in \eqref{eq:g2}, which bounds the expected deviation 
from the $n$th constraint during the first $K$ frames (recall \eqref{eq:parta} from Lemma 
 \ref{lem:virtual-queues}). Under a mild additional ``Slater-type'' assumption that ensures 
 all constraints can be satisfied with ``$\epsilon$-slackness,'' 
 a stronger result on the virtual queues can be shown, namely, that the same
 algorithm yields queues with average size $O(V)$ \cite{sno-text}.  This typically ensures a tighter constraint tradeoff 
 than that given in \eqref{eq:g2}.
 A related improved tradeoff is explored in more detail in Section  \ref{section:random-task-arrivals}.

\begin{proof} (Theorem \ref{thm:perf1} part (a)) 
Given $\bv{Q}[k]$ for frame $k$, our control decisions minimize the last two terms in the right-hand-side of the 
drift-plus-penalty ratio bound \eqref{eq:dpp1}, and hence: 
\begin{eqnarray}
\hspace{-.3in}\frac{\expect{\Delta[k] + Ve[k]|\bv{Q}[k]}}{\expect{D[k] + I[k]|\bv{Q}[k]}} &\leq& \frac{B}{\expect{D[k]+I[k]|\bv{Q}[k]}}
+ \frac{\expect{V\hat{e}(c^*[k],m^*[k])|\bv{Q}[k]}}{\expect{\hat{D}(c^*[k], m^*[k]) + I^*[k]|\bv{Q}[k]}} \nonumber \\
&& + \frac{\sum_{n=1}^NQ_n[k]\expect{\lambda_n(\hat{D}(c^*[k],m^*[k]) + I^*[k]) - 1_n^*[k]|\bv{Q}[k]}}{\expect{\hat{D}(c^*[k], m^*[k]) + I^*[k]|\bv{Q}[k]}} \label{eq:dpp1-proof} 
\end{eqnarray}
where $c^*[k]$, $m^*[k]$, $I^*[k]$, $1_n^*[k]$ are from any alternative (possibly randomized) decisions that
can be made on frame $k$. Now recall the existence of stationary and randomized decisions that
yield \eqref{eq:sr1}-\eqref{eq:sr2}.  In particular, these decisions are independent of queue backlog
$\bv{Q}[k]$ and thus yield (from \eqref{eq:sr1}): 
\begin{eqnarray*}
\frac{\expect{\hat{e}(c^*[k], m^*[k])|\bv{Q}[k]}}{\expect{\hat{D}(c^*[k], m^*[k]) + I^*[k]|\bv{Q}[k]}} = \frac{\expect{\hat{e}(c^*[k], m^*[k])}}{\expect{\hat{D}(c^*[k], m^*[k]) + I^*[k]}} 
= power^{opt}  
\end{eqnarray*}
and for all $n \in \{1, \ldots, N\}$ we have (from \eqref{eq:sr2}): 
\begin{eqnarray*}
\expect{\lambda_n(\hat{D}(c^*[k], m^*[k])+I^*[k]) - 1_n^*[k]|\bv{Q}[k]} = \expect{\lambda_n(\hat{D}(c^*[k], m^*[k])+I^*[k]) - 1_n^*[k]} \leq 0
\end{eqnarray*}
Plugging the above into the right-hand-side of \eqref{eq:dpp1-proof} yields: 
\begin{eqnarray*}
\frac{\expect{\Delta[k] + Ve[k]|\bv{Q}[k]}}{\expect{D[k] + I[k]|\bv{Q}[k]}} &\leq& \frac{B}{\expect{D[k]+I[k]|\bv{Q}[k]}}
+ Vpower^{opt} 
\end{eqnarray*}
Rearranging terms gives: 
\begin{eqnarray*}
\expect{\Delta[k] + Ve[k]|\bv{Q}[k]} \leq B + Vpower^{opt} \expect{D[k] + I[k]|\bv{Q}[k]} 
\end{eqnarray*}
Taking expectations of the above (with respect to the random $\bv{Q}[k]$) and using the law
of iterated expectations gives: 
\begin{eqnarray}
\expect{\Delta[k] + Ve[k]} \leq B + Vpower^{opt} \expect{D[k] + I[k]} \label{eq:gives} 
\end{eqnarray}
The above holds for all $k \in \{0, 1, 2, \ldots\}$.  Fixing a positive integer $K$ and summing \eqref{eq:gives} 
over $k \in \{0, 1, \ldots, K-1\}$ yields, by the definition  $\Delta[k] = L[k+1]-L[k]$: 
\[ \expect{L[K] - L[0]} + V\sum_{k=0}^{K-1}\expect{e[k]} \leq BK + Vpower^{opt} \sum_{k=0}^{K-1} \expect{D[k]+I[k]} \]
Noting that $L[0]=0$ and $L[K]\geq 0$ and using the definitions of  
$\overline{e}[K]$, $\overline{D}[K]$, $\overline{I}[K]$ yields: 
\[ VK\overline{e}[K] \leq BK + VKpower^{opt}(\overline{D}[K]+\overline{I}[K]) \]
Rearranging terms yields the result of part (a). 
\end{proof} 

\begin{proof} (Theorem \ref{thm:perf1} part (b)) 
To prove part (b), note from \eqref{eq:gives} we have: 
\[ \expect{\Delta[k]} \leq B + V\expect{power^{opt}(D[k] + I[k]) - e[k]} \leq B + V\beta \]
Summing the above over $k \in \{0, 1, \ldots, K-1\}$ gives: 
\[ \expect{L[K]} - \expect{L[0]} \leq (B + V\beta)K \]
Using the definition of $L[K]$ and noting that $L[0]=0$ gives: 
\[ \sum_{l=1}^K\expect{Q_n[K]^2} \leq 2(B+V\beta)K \]
Thus, we have $\expect{\norm{\bv{Q}[K]}^2} \leq 2(B+V\beta)K$. Jensen's inequality for $f(x) = x^2$ 
ensures
$\expect{\norm{\bv{Q}[k]}}^2 \leq \expect{\norm{\bv{Q}[k]}^2}$, and so for all positive integers $K$: 
\begin{equation} \label{eq:blah}  
\expect{\norm{\bv{Q}[k]}}^2 \leq 2(B+V\beta)K 
\end{equation} 
Taking a square root of both sides of \eqref{eq:blah} and dividing by $K$ 
proves \eqref{eq:g2}.  From \eqref{eq:g2} we have for each $n \in \{1, \ldots, N\}$: 
\[ \lim_{K\rightarrow\infty} \frac{\expect{Q_n[k]}}{K} \leq \lim_{K\rightarrow\infty} \frac{\expect{\norm{\bv{Q}[K]}}}{K} \leq \lim_{K\rightarrow\infty} \frac{\sqrt{2(B+V\beta)}}{\sqrt{K}} = 0 \]
and hence by Lemma \ref{lem:virtual-queues} we know constraint \eqref{eq:desired2} holds. 
 Further, in \cite{lyap-opt2} it is shown that \eqref{eq:blah} together with the fact that second moments
of queue changes are bounded  implies $\lim_{k\rightarrow\infty}Q_n[k]/k=0$ with probability 1.  Thus, 
 \eqref{eq:desired1} holds. 
\end{proof} 

\subsection{Simulation}  \label{section:sim1} 

We first simulate the task scheduling 
algorithm for the simple deterministic system with one class and one constraint, as described in 
Section \ref{section:one-class-example}.  The $\hat{e}(m)$ and $\hat{D}(m)$ functions 
are defined in \eqref{eq:m-value1}-\eqref{eq:m-value2}, and the goal is to minimize average
power subject to a processing rate constraint $1/(\overline{D} + \overline{I}) \geq 0.2$.  
We already know the optimal power is $power^{opt} = 7/15 \approx 0.466667$.
We expect the algorithm to approach this optimal power as $V$ is increased, and to approach 
the desired behavior of using $I[k]=0$ for all $k$, meeting the constraint with equality, 
and using $m[k]=1$ for $1/3$ of the frames.  This is indeed what happens, although in this
simple case the algorithm seems insensitive to the $V$ parameter and 
locks into a desirable periodic schedule even for very low (but positive) $V$ values. 
Using $V=1$ and one million
frames, the algorithm gets average power $0.466661$, uses $m[k]=1$ a fraction of time $0.333340$, 
has average idle time $\overline{I} = 0.000010$, and yields a processing rate $0.199999$ (almost
exactly equal to the desired constraint of $0.2$).  Increasing $V$ yields similar performance.  
The constraint is still satisfied when we decrease
the value of $V$, but average power degrades (being $0.526316$ for $V=0$). 

\begin{figure}[cht]
\centering
\begin{minipage}{2.6in}  
\includegraphics[height=2.1in, width=2.7in]{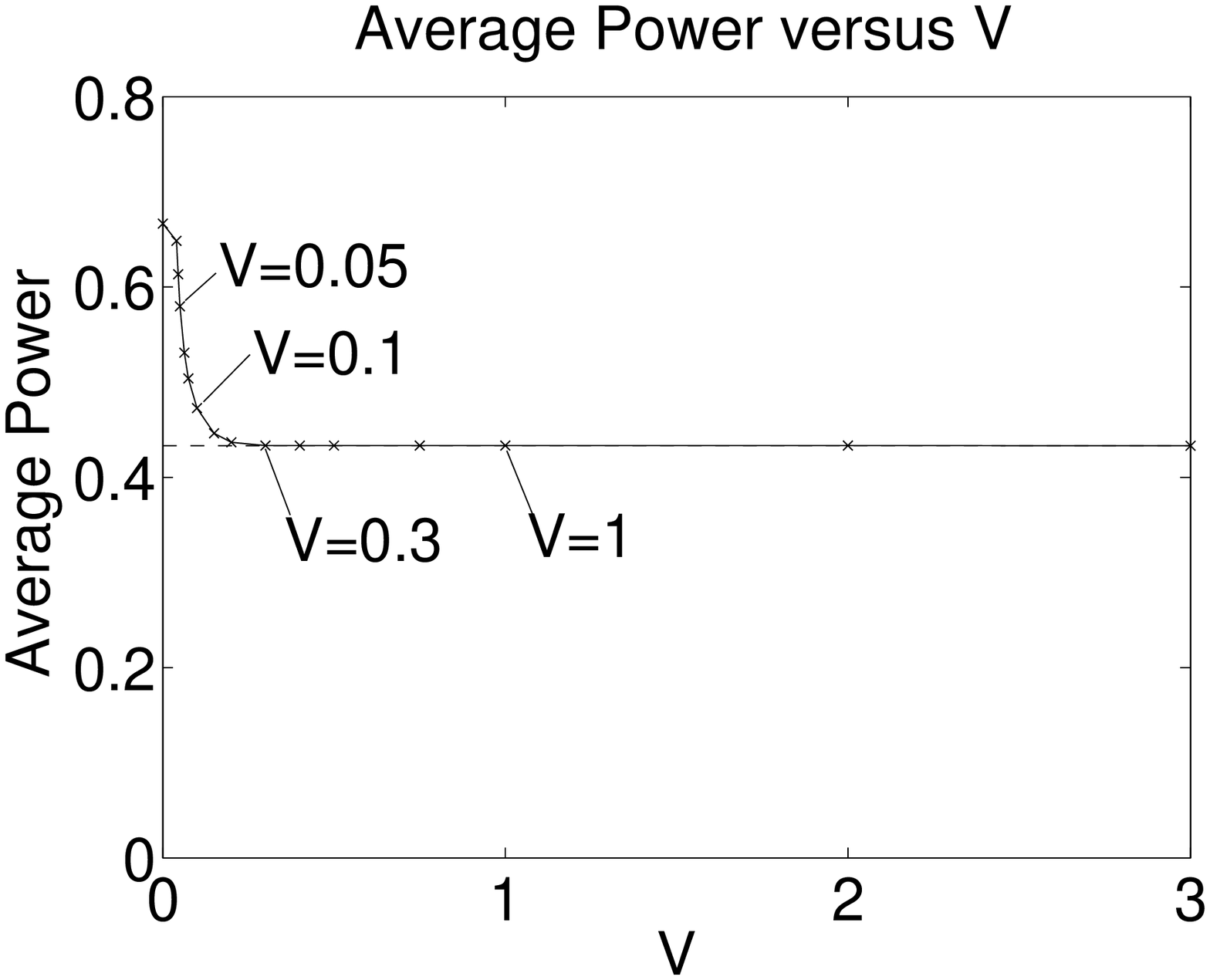}
\caption{Average power versus $V$.}
\label{fig:plot1}
\end{minipage}
\qquad
\qquad
\begin{minipage}{2.6in}
  \includegraphics[height=2.1in, width=2.7in]{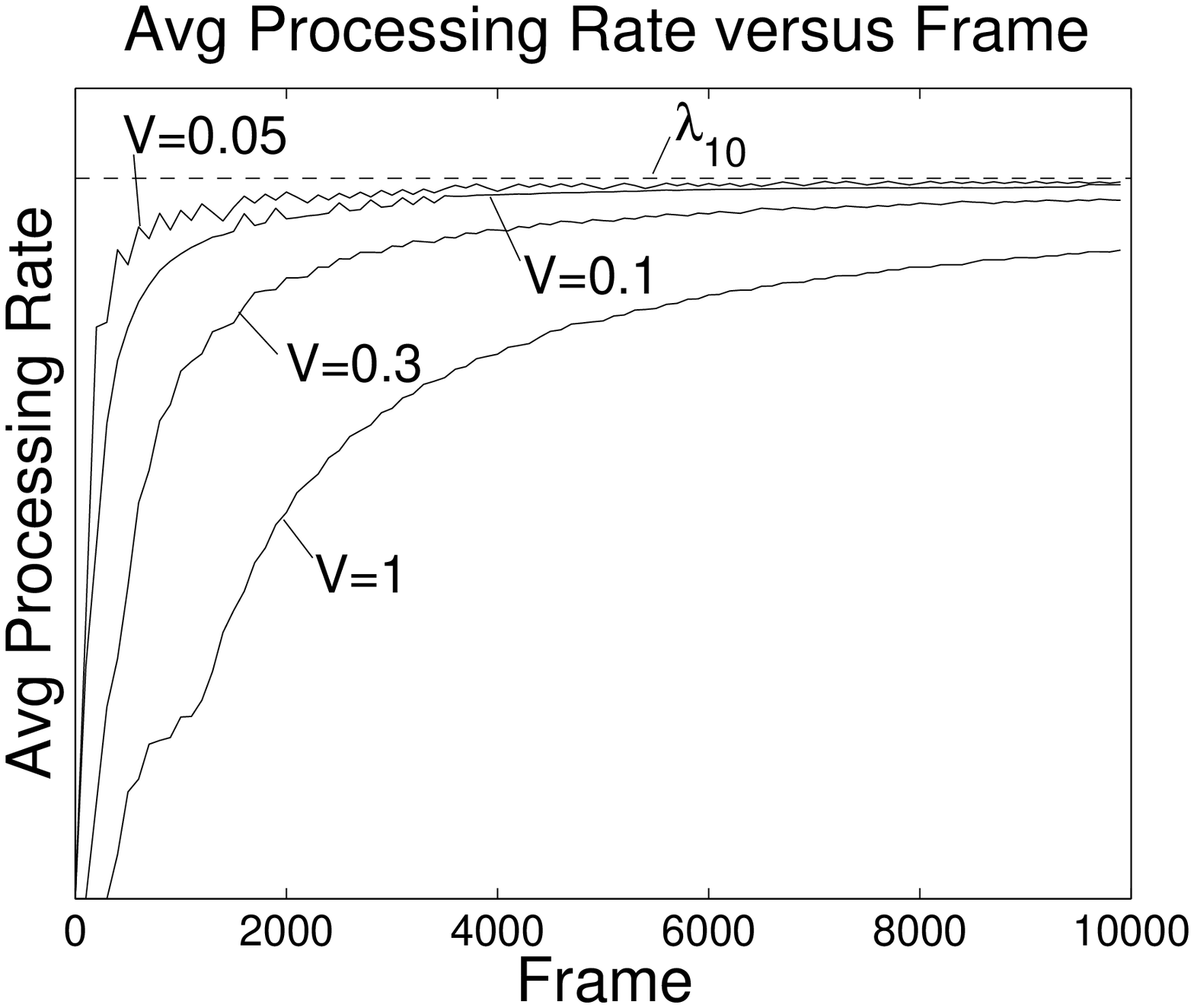} 
  \caption{Processing Rate $\overline{1}_{10}/\overline{T}$ versus frame index.}
  \label{fig:plot2}
\end{minipage}
\end{figure}

We next consider a system with 10 classes of tasks and two processing modes. 
 The energy and delay characteristics for each class $i \in \{1, \ldots, 10\}$ and mode $m \in \{1, 2\}$ 
are: 
\begin{eqnarray}
\mbox{Mode 1:} & (\hat{e}(i, 1), \hat{D}(i, 1)) = (1i, 5i)  \label{eq:ehatrand1} \\
\mbox{Mode 2:} & (\hat{e}(i,2), \hat{D}(i,2)) = (2i, 3i) \label{eq:ehatrand2} 
\end{eqnarray}
so that mode 1 uses less energy but takes longer than mode 2, and the computational
requirements for each class increase with $i \in \{1, \ldots, 10\}$.  We assume 
desired rates are given by $\lambda_i = \rho/(30i)$ for $i \in \{1, \ldots, 10\}$, for some positive value $\rho$.  
The problem is feasible 
whenever $\rho \leq 1$.  We use $\rho = 0.8$ and run the simulation for 10 million frames. 
Fig. \ref{fig:plot1} shows the resulting average power as $V$ is varied between $0$
and $3$, which converges to near optimal after $V=0.3$.   All 10  processing rate constraints are met within
5 decimal points of accuracy after the 10 million frames.  An illustration of how convergence time is affected
by the $V$ parameter is shown in Fig. \ref{fig:plot2}, which illustrates the average processing rate 
$\overline{1}_{10}/\overline{T}$ 
for class 10,  as compared to the desired constraint $\lambda_{10} = \rho/300$.
It is seen, for example, that convergence is faster for $V=0.05$ than for $V=1$.  Convergence times can be 
improved using non-zero initial queue backlog and the theory of \emph{place holder backlog} in \cite{sno-text}, 
although we omit this topic for brevity.

\section{Optimization with General Attributes} \label{section:general-attributes} 

This section generalizes the problem to allow time average 
optimization for abstract \emph{attributes}.  Consider again a frame-based system with frame
index $k \in \{0, 1, 2, \ldots\}$.   Every frame $k$, the controller makes a \emph{control action} $\alpha[k]$, 
chosen within an abstract set $\script{A}$ of allowable actions.  The action $\alpha[k]$ affects
the \emph{frame size} $T[k]$ and an \emph{attribute vector} 
$\bv{y}[k] = (y_0[k], y_1[k], \ldots, y_L[k])$.  Specifically, assume these are random functions that
are conditionally independent of the past given the current $\alpha[k]$ decision, with mean values
given by functions $\hat{T}(\alpha)$ and $\hat{y}_l(\alpha)$ for all $\alpha \in \script{A}$: 
\[ \hat{T}(\alpha) = \expect{T[k]|\alpha[k]=\alpha} \: \: , \: \: \hat{y}_l(\alpha) = \expect{y_l[k]|\alpha[k]=\alpha} \]
Similar to the previous section, it is assumed there is a minimum 
frame size $T_{min}>0$ such that $T[k] \geq T_{min}$
for all $k$, and that second moments are bounded by a constant $\sigma^2$, regardless of the policy $\alpha[k]$. 
The joint distribution of $(T[k], y_0[k], y_1[k], \ldots, y_L[k])$ is otherwise arbitrary. 

Define frame averages $\overline{T}$ and $\overline{y}_l$ by: 
\[ \overline{T} = \lim_{K\rightarrow\infty}\frac{1}{K}\sum_{k=0}^{K-1} T[k] \: \: \: , \: \: \: \overline{y}_l = \lim_{K\rightarrow\infty} \frac{1}{K}\sum_{k=0}^{K-1}y_l[k] \]
As discussed in Section \ref{section:time-average-as-frame-average}, 
the value $\overline{y}_l/\overline{T}$ represents the \emph{time average} associated with attribute $y_l[k]$. 
The general problem is then: 
\begin{eqnarray} 
\mbox{Minimize:} && \overline{y}_0/\overline{T} \label{eq:p0} \\
\mbox{Subject to:} && \overline{y}_l/\overline{T} \leq c_l \: \: \: \forall l \in \{1, \ldots, L\} \label{eq:p1} \\
&& \alpha[k] \in \script{A} \: \: \: \forall k \in \{0, 1, 2, \ldots\} \label{eq:p2}  
\end{eqnarray}
where $c_1, \ldots, c_L$ are given constants that specify the desired time average constraints. 

\subsection{Mapping to the Task Scheduling Problem} 

To illustrate the generality of this framework, this subsection uses the new 
notation to exactly represent the task scheduling problem from Section \ref{section:task-scheduling}.   
For that problem, one can define the control action $\alpha[k]$ to have the form
$\alpha[k] = (c[k], m[k], I[k])$, and the action space $\script{A}$ is then the
set of all $(c,m,I)$ such that $c \in \{1, \ldots, N\}$, $m \in \script{M}$, and $0 \leq I \leq I_{max}$. 

The frame size is $T[k] = D[k] + I[k]$, and $\hat{T}(\alpha[k])$ is given by: 
\[ \hat{T}(\alpha[k]) = \hat{D}(c[k],m[k]) + I[k] \]
We then define $y_0[k]$ as the energy expended in frame $k$, so that $y_0[k] = e[k]$ and 
$\hat{y}_0(\alpha[k]) = \hat{e}(c[k], m[k])$.  There are $N$ constraints, so define $L=N$. 
To express the desired constraints $\overline{1}_n/\overline{T} \geq \lambda_n$ in the form $\overline{y}_n/\overline{T} \leq c_n$, one  
can define $y_n[k] = -1_n[k]$ and $c_n = -\lambda_n$ for each $n \in \{1, \ldots, N\}$,  
and $\hat{y}_n(\alpha[k]) = -1_n[k]$.  Alternatively, one 
could define $y_n[k] = \lambda_nT[k] - 1_n[k]$ and enforce the constraint $\overline{y}_n \leq 0$ for all $n \in \{1, \ldots, N\}$. 

This general setup provides more flexibility.  For example, suppose the idle state does not 
use 0 energy, but operates at a low power $p_{idle}$ and expends total energy $p_{idle}I[k]$ on frame
$k$.  Then total energy for frame $k$ 
can be defined as $y_0[k] = e[k] + p_{idle}I[k]$, where $e[k]$ is the
energy spent in the busy period.  The setup can also handle systems with multiple idle mode options, 
each providing a different energy savings but 
incurring a different wakeup time. 
%As another extension, we can allow the action
%space $\script{A}$ to include the possibility of choosing \emph{no} task in a busy period, taking some small
%amount of time, to allow the possibility of skipping some busy periods to save additional energy. 
%More substantial extensions that allow randomly arriving tasks,  include \emph{quality} 
%attributes and average power constraints, and solve general linear fractional problems, are considered
%in Section \ref{section:randomly-arriving-tasks} and in the Exercises. 

\subsection{The General Algorithm} 

The algorithm for solving the general problem \eqref{eq:p0}-\eqref{eq:p2} is described below. 
Each constraint $\overline{y}_l \leq c_l\overline{T}$ in \eqref{eq:p1} is treated using a virtual
queue $Q_l[k]$ with update: 
\begin{equation} \label{eq:z-update} 
Q_l[k+1] = \max[Q_l[k] + y_l[k] - c_lT[k], 0] \: \: \forall l \in \{1, \ldots, L\} 
\end{equation} 
Defining $L[k]$ and $\Delta[k]$ as before (in Section \ref{section:dppr}) leads to the following 
bound on the drift-plus-penalty ratio, which can be proven in a manner similar to Lemma \ref{lem:dpp-bound1}: 
\begin{eqnarray}
\frac{\expect{\Delta[k] + Vy_0[k]|\bv{Q}[k]}}{\expect{T[k]|\bv{Q}[k]}} \leq \frac{B}{\expect{T[k]|\bv{Q}[k]}}
+ \frac{\expect{V\hat{y}_0(\alpha[k]) + \sum_{l=1}^LQ_l[k]\hat{y}_l(\alpha[k])|\bv{Q}[k]}}{\expect{\hat{T}(\alpha[k])|\bv{Q}[k]}}  \label{eq:dpp} 
\end{eqnarray}
where $B$ is a constant that satisfies the following for all $\bv{Q}[k]$ and all possible actions $\alpha[k]$: 
\[ B \geq \frac{1}{2}\sum_{l=1}^L\expect{(y_l[k] - c_lT[k])^2|\bv{Q}[k]} \]

Every frame, the controller observes queues $\bv{Q}[k]$ and 
 takes an action $\alpha[k] \in \script{A}$ that minimizes the second term on the right-hand-side of \eqref{eq:dpp}. 
We know from Lemma \ref{lem:deterministic-min} that minimizing the ratio of expectations
is accomplished by a deterministic selection of $\alpha[k] \in \script{A}$.  The resulting algorithm is: 

\begin{itemize} 
\item Observe $\bv{Q}[k]$ and choose $\alpha[k] \in \script{A}$ to minimize (breaking ties arbitrarily): 
\begin{equation} \label{eq:new-min} 
 \frac{V\hat{y}_0(\alpha[k]) + \sum_{l=1}^LQ_l[k]\hat{y}_l(\alpha[k])}{\hat{T}(\alpha[k])} 
 \end{equation} 

\item Update virtual queues $Q_l[k]$ for each $l \in \{1, \ldots, L\}$ via \eqref{eq:z-update}.
\end{itemize} 

One subtlety is that the expression \eqref{eq:new-min} may not have an achievable minimum over
the general (possibly infinite) set $\script{A}$ (for example, the infimum of the function $f(x) = x$ over the 
open interval $0 < x < 1$ is not achievable over that interval).  This is 
no problem:  Our algorithm in fact works for any \emph{approximate} minimum that is an additive constant
$C$ away from the exact infimum every frame $k$ 
(for any arbitrarily large constant $C\geq0$).  This 
effectively changes the ``$B$''  constant in our performance bounds 
to a new constant ``$B+C$'' \cite{sno-text}. 
Let $ratio^{opt}$ represent the optimal ratio of $\overline{y}_0/\overline{T}$ for the problem
\eqref{eq:p0}-\eqref{eq:p2}.  As before, it can be shown that if the problem is feasible (so that there exists
an algorithm that can achieve the constraints
\eqref{eq:p1}-\eqref{eq:p2}), then any $C$-additive approximation of the above algorithm satisfies all 
desired constraints and yields $\overline{y}_0/\overline{T} \leq ratio^{opt} + O(1/V)$, which can be 
pushed arbitrarily close to $ratio^{opt}$ as $V$ is increased, with the same tradeoff in the queue sizes 
(and hence convergence times) with $V$.  The proof of this is similar to that of Theorem \ref{thm:perf1}, 
and is omitted for brevity (see \cite{sno-text}\cite{renewals-asilomar2010} for the full proof). 

\subsection{Random Task Arrivals and Flow Control} \label{section:random-task-arrivals}

\begin{figure}[htbp]
   \centering
   \includegraphics[height=1.7in,width=5in]{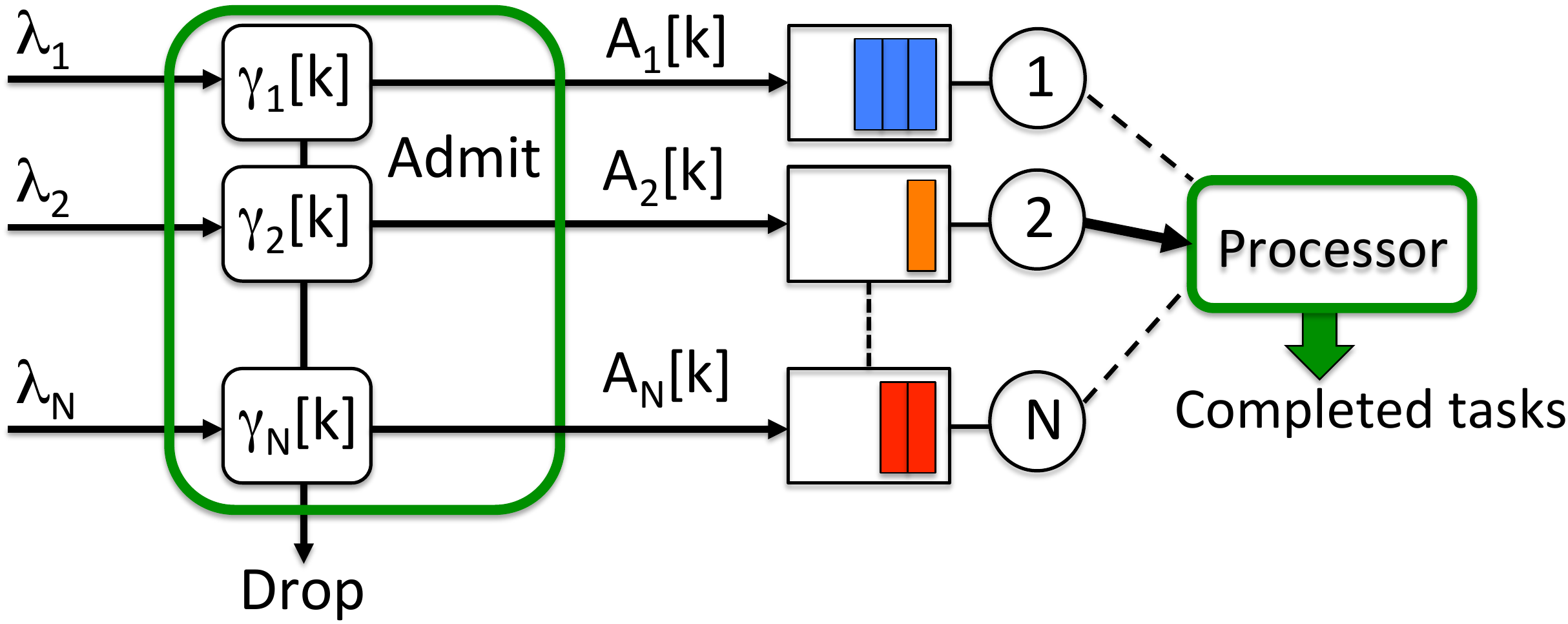} % requires the graphicx package
   \caption{A task processing system with random arrivals and flow control.}
   \label{fig:random-arrivals}
\end{figure}

Again consider a system with $N$ classes of tasks, as in Section \ref{section:task-scheduling}. Each 
frame $k$ again has a busy period of duration $D[k]$ and an idle period of duration $I[k]$ as in 
Fig. \ref{fig:timeline}.  However, rather than 
always having tasks available for processing, this subsection assumes tasks arrive \emph{randomly}
with rates $(\lambda_1, \ldots, \lambda_N)$, where $\lambda_n$ is the rate of task arrivals per unit time (see
Fig. \ref{fig:random-arrivals}). 
At the beginning of each busy period, the controller chooses a variable $c[k]$ that specifies which type of task is
performed. However, $c[k]$ can now take values in the set $\{0, 1, \ldots, N\}$, where $c[k] =0$ is a \emph{null} 
choice that selects no task on frame $k$.  If $c[k]=0$, the busy period has some positive size $D_{0}$ and 
may spend a small amount of 
energy to power the electronics, but does not process any task. 
The mode selection variable $m[k]$ takes values in the 
same set $\script{M}$ as before. 
The idle time variable $I[k]$ is again chosen in the interval $[0, I_{max}]$. 

Further, for each $n \in \{1, 2, \ldots, N\}$ 
we introduce \emph{flow control variables} $\gamma_n[k]$, chosen in the interval $0 \leq \gamma_n[k] \leq 1$. 
The variable $\gamma_n[k]$ represents the probability of admitting each new randomly arriving task of 
class $n$ on frame $k$ (see Fig. \ref{fig:random-arrivals}).  This enables the system to drop tasks if the raw arrival 
rates $(\lambda_1, \ldots, \lambda_N)$
cannot be supported.  Let $\bv{\gamma}[k] = (\gamma_1[k], \ldots, \gamma_N[k])$ be the vector of these
variables. 

We thus have $\alpha[k] = (c[k], m[k], I[k], \bv{\gamma}[k])$, with action space $\script{A}$ being the set of all $(c,m,I, \bv{\gamma})$
such that $c \in \{0, 1, \ldots, N\}$, $m \in \script{M}$, $0 \leq I \leq I_{max}$, and $0 \leq \gamma_n \leq 1$ for 
$n \in \{1, \ldots, N\}$. 
Define $e[k]$ and $D[k]$ as the energy and busy period duration for frame $k$.  Assume $e[k]$ 
depends only on $(c[k], m[k], I[k])$, and $D[k]$ depends only on $(c[k], m[k])$, with 
averages given by functions $\hat{e}(c[k], m[k], I[k])$ and 
$\hat{D}(c[k], m[k])$: 
\[ \hat{e}(c[k], m[k], I[k]) = \expect{e[k]|c[k], m[k], I[k]} \: \: \: , \: \: \: \hat{D}(c[k], m[k]) = \expect{D[k]|c[k], m[k]} \]
Finally, for each $n \in \{1, \dots, N\}$, 
define $A_n[k]$ as the random number of new arrivals admitted on frame $k$, which depends on the 
total frame size $D[k]+I[k]$ and the admission probability $\gamma_n[k]$.  
Formally, assume the arrival vector 
$(A_1[k], \ldots, A_N[k])$ is conditionally independent of the past
given the current $\alpha[k]$ used, with expectations: 
\begin{equation} \label{eq:ahat} 
\expect{A_n[k]| \alpha[k]} = \lambda_n\gamma_n[k][\hat{D}(c[k], m[k]) +I[k]] 
\end{equation} 
The assumption on independence of the past holds whenever arrivals are independent and Poisson, 
or when all frame sizes are an integer number of fixed size slots, and arrivals are independent and
identically distributed (i.i.d.) over slots with 
some general distribution.

We seek to maximize a weighted sum of admission rates subject to supporting all of the admitted
tasks, and to maintaining average power to within a given positive constant $P_{av}$: 
\begin{eqnarray}
\mbox{Maximize:} &&  \frac{\sum_{n=1}^Nw_n\overline{A}_n}{\overline{D} + \overline{I}} \label{eq:rand1} \\
\mbox{Subject to:} && \overline{A}_n/(\overline{D} + \overline{I}) \leq \overline{1}_n/(\overline{D} + \overline{I})  \: \: \: \forall n \in \{1, \ldots, N\} \label{eq:rand2} \\
&& \frac{\overline{e}}{\overline{D} + \overline{I}} \leq P_{av} \label{eq:rand3} \\
&& \alpha[k] \in \script{A} \: \: \: \forall k \in \{0, 1, 2, \ldots\} \label{eq:rand4} 
\end{eqnarray}
where $(w_1, \ldots, w_N)$ are a collection of positive weights that prioritize the different classes in the 
optimization objective. 

We thus have $L=N+1$ constraints.  To treat this problem, define $T[k] = D[k] + I[k]$, so that $\hat{T}(\alpha[k]) = \hat{D}(c[k],m[k]) + I[k]$.  Further define  $y_0[k]$, $y_n[k]$ for $n \in \{1, \ldots, N\}$, 
and $x[k]$ as:  
\begin{eqnarray*}
y_0[k] = -\sum_{n=1}^Nw_nA_n[k]  &\implies&  \hat{y}_0(\alpha[k]) = -\sum_{n=1}^Nw_n\lambda_n\gamma_n[k][\hat{D}(c[k], m[k]) + I[k]] \\
y_n[k] = A_n[k] - 1_n[k] &\implies&  \hat{y}_n(\alpha[k]) = \lambda_n\gamma_n[k][\hat{D}(c[k], m[k]) + I[k]] - 1_n[k]  \\
 x[k] = e[k] - [D[k]+I[k]]P_{av} &\implies& \hat{x}(\alpha[k]) = \hat{e}(c[k], m[k], I[k]) - [\hat{D}(c[k], m[k])+ I[k]]P_{av}
\end{eqnarray*}
Then: 
\begin{itemize} 
\item Minimizing $\overline{y}_0/\overline{T}$ is equivalent to \eqref{eq:rand1}.
\item The constraints $\overline{y}_n \leq 0$ for $n \in \{1, \ldots, N\}$ are equivalent to \eqref{eq:rand2}. 
\item The constraint $\overline{x} \leq 0$ is equivalent to \eqref{eq:rand3}. 
\end{itemize} 

Note that the above problem does not specify any explicit queueing for the randomly arriving 
tasks.  The algorithm will in fact construct explicit queues (so that the virtual queues can be viewed
as actual queues).  Note also that the 
constraint $\alpha[k] \in \script{A}$ does not allow restrictions on actions based on the queue
state, such as when the queue is empty or not.  Thus, in principle, we allow the possibility of 
``processing'' a task of class $n$ even when there is no such task available.  In this case, we
assume this processing is still costly, in that it 
incurs time equal to $\hat{D}(c[k], m[k])$ and energy equal to $\hat{e}(c[k],m[k],I[k])$.  Our
algorithm will naturally learn to avoid the inefficiencies associated with such actions. 

\subsubsection{The Dynamic Algorithm for Random Task Arrivals} 

To enforce the constraints $\overline{y}_n \leq 0$ for each $n \in \{1, \ldots, N\}$, define queue $Q_n[k]$
with update: 
\begin{equation} \label{eq:zrand-update} 
Q_n[k+1] = \max[Q_n[k] + A_n[k] - 1_n[k], 0]  
\end{equation} 
To enforce $\overline{x} \leq 0$, define a virtual queue $Z[k]$ with update: 
\begin{equation} \label{eq:h-update} 
Z[k+1] = \max[Z[k] + e[k] - [D[k] + I[k]]P_{av}, 0] 
\end{equation} 
It can be seen that the queue update \eqref{eq:zrand-update} is the same as that of an \emph{actual queue} 
for class $n$ tasks, with random task arrivals $A_n[k]$ and task service $1_n[k]$. 
The minimization of \eqref{eq:new-min} then becomes the following: Every frame $k$, 
observe queues $\bv{Q}[k]$ and $Z[k]$.  Then choose $c[k] \in \{0, 1, \ldots, N\}$, $m[k] \in \script{M}$, 
$I[k] \in [0, I_{max}]$, and $\gamma_n[k] \in [0,1]$ for all $n \in \{1, \ldots, N\}$ to minimize: 
\begin{eqnarray}
\frac{-V\sum_{n=1}^Nw_n\lambda_n\gamma_n[k][\hat{D}(c[k], m[k]) + I[k]]}{\hat{D}(c[k],m[k]) + I[k]} \nonumber \\
 + \frac{\sum_{n=1}^NQ_n[k](\lambda_n\gamma_n[k][\hat{D}(c[k],m[k]) + I[k]] - 1_n[k])}{\hat{D}(c[k],m[k]) + I[k]} \nonumber \\
+ \frac{Z[k](\hat{e}(c[k], m[k], I[k]) - [\hat{D}(c[k], m[k]) + I[k]]P_{av})}{\hat{D}(c[k], m[k]) + I[k]} \label{eq:huge} 
\end{eqnarray}

After a simplifying cancellation of terms, it is easy to see that the $\gamma_n[k]$ decisions can be
separated from all other decisions (see Exercise \ref{ex:separate-gamma}).  The resulting 
algorithm then observes queues $\bv{Q}[k]$ and $Z[k]$ every frame $k$ and performs the following: 
\begin{itemize} 
\item (Flow Control) For each $n \in \{1, \ldots, N\}$, choose $\gamma_n[k]$ as: 
\begin{equation} \label{eq:optimal-gamma} 
 \gamma_n[k] = \left\{ \begin{array}{ll}
                          1 &\mbox{ if $Q_n[k] \leq Vw_n$} \\
                             0  & \mbox{ otherwise} 
                            \end{array}
                                 \right. 
 \end{equation} 
\item (Task Scheduling) Choose $c[k] \in \{0, 1, \ldots, N\}$, $m[k] \in \script{M}$, $I[k] \in [0, I_{max}]$ to minimize: 
\begin{equation} \label{eq:other-dec} 
 \frac{Z[k]\hat{e}(c[k], m[k], I[k])  -\sum_{n=1}^NQ_n[k]1_n[k]}{\hat{D}(c[k], m[k]) + I[k]} 
 \end{equation} 
\item (Queue Update) Update $Q_n[k]$ for each $n \in \{1, \ldots, N\}$ by \eqref{eq:zrand-update}, 
and update $Z[k]$ by \eqref{eq:h-update}. 
\end{itemize} 

The minimization problem \eqref{eq:other-dec} is similar to  \eqref{eq:minimize}, 
and can be carried out in the same manner as discussed in Section \ref{section:dppr}. A key observation about the
above algorithm is that \emph{it does not require knowledge of the arrival rates 
$(\lambda_1, \ldots, \lambda_N)$}.  Indeed, the $\lambda_n$ terms cancel out of the minimization, 
so that the flow control variables $\gamma_n[k]$ in \eqref{eq:optimal-gamma} 
make ``bang-bang'' decisions that admit all newly arriving tasks of class $n$ on frame $k$
if $Q_n[k] \leq Vw_n$, and 
admit none otherwise.    This property makes the algorithm naturally adaptive to situations when
the arrival rates change, as shown in the simulations of Section \ref{section:simulation-random}. 

Note that if $\hat{e}(0,m,I) < \hat{e}(c,m,I)$ for all $c \in \{1, \ldots, N\}$, 
$m \in \script{M}$, $I\in[0, I_{max}]$, so that the energy associated with processing \emph{no task} 
is less than the energy of processing any class $c\neq0$, then the minimization in \eqref{eq:other-dec} 
will never select a class $c$ such that $Q_c[k]=0$.  That is, the algorithm naturally will   
never select a class for which no task is available. 

\subsubsection{Deterministic Queue Bounds and Constraint Violation Bounds} \label{section:det-bound} 

In addition to satisfying the desired constraints and achieving a weighted sum of admitted rates
that is within $O(1/V)$ of optimality, the flow control structure of the task scheduling algorithm admits
\emph{deterministic queue bounds}.  Specifically, assume all frame sizes are bounded by a constant $T_{max}$, 
and that the raw number of class $n$ arrivals per frame (before admission control) is at most $A_{n,max}$. 
By the flow control policy \eqref{eq:optimal-gamma}, new arrivals of class $n$ are only admitted
if $Q_n[k] \leq Vw_n$.  Thus, assuming that $Q_n[0] \leq Vw_n + A_{n,max}$, 
we must have $Q_n[k] \leq Vw_n + A_{n,max}$
for all frames $k \in \{0, 1, 2, \ldots\}$.  This specifies a worst-case queue backlog that is $O(V)$, which 
establishes an explicit $[O(1/V), O(V)]$ performance-backlog tradeoff that is superior to that given in 
\eqref{eq:g2}. 

With mild additional structure on the $\hat{e}(c,m,I)$ function, the deterministic bounds 
on queues $Q_n[k]$ lead to a deterministic bound on $Z[k]$, so that one can compute a value
$Z_{max}$, of size $O(V)$, such that $Z[k] \leq Z_{max}$ for all $k$. This is explored in Exercise \ref{ex:deterministic-queue} (see also \cite{neely-energy-it}).

\subsection{Simulations and Adaptiveness of Random Task Scheduling} \label{section:simulation-random}

Here we simulate the dynamic task scheduling and flow control algorithm \eqref{eq:optimal-gamma}-\eqref{eq:other-dec}, using the 10-class system model defined in Section \ref{section:sim1} with $\hat{e}(c,m)$ functions
given in \eqref{eq:ehatrand1}-\eqref{eq:ehatrand2}.  For consistency with that model, we remove
the $c[k]=0$ option, so that the decision \eqref{eq:ehatrand2}  
chooses $c[k]=1, m[k]=1$, incurring one unit of energy, in 
case no tasks are available.   Arrivals are from independent Bernoulli processes with rates $\lambda_i = \rho/(30i)$
for each class $i \in \{1, \ldots, 10\}$, with $\rho=0.8$. 
We use weights $w_n=1$ for all $n$, so that the objective is
to maximize total throughput, and $P_{av} = 0.5$, which we know is feasible
from results in Fig. \ref{fig:plot1} of 
Section \ref{section:sim1}.  Thus, we expect the algorithm to learn to admit everything, so that 
the admission rate approaches the total arrival rate as $V$ is increased.  We simulate for 10 million frames, using $V$
in the interval from $0$ to $200$.  Results are shown in Figs. 
\ref{fig:rand1} and  \ref{fig:rand2}.  Fig \ref{fig:rand1} shows the algorithm learns to admit everything for large $V$ (100 or above), 
and Fig. \ref{fig:rand2} plots the resulting average queue size (in number of tasks) per queue, together with the 
deterministic bound $Q_n[k] \leq V + 60$ (where $A_{n,max} = 60$ in this case because there is at most one arrival per slot, 
and the largest possible frame is $D_{max} + I_{max} = 50 + 10 = 60$).  The average power constraint was 
satisfied (with slackness) for all cases.  The average queue size in Fig. \ref{fig:rand2} grows linearly with $V$ until $V=100$, when it saturates
by admitting everything.  The saturation value is the average queue size associated with admitting the raw arrival rates directly. 

\begin{figure}[cht]
\centering
\begin{minipage}{2.6in}  
\includegraphics[height=2.1in, width=2.7in]{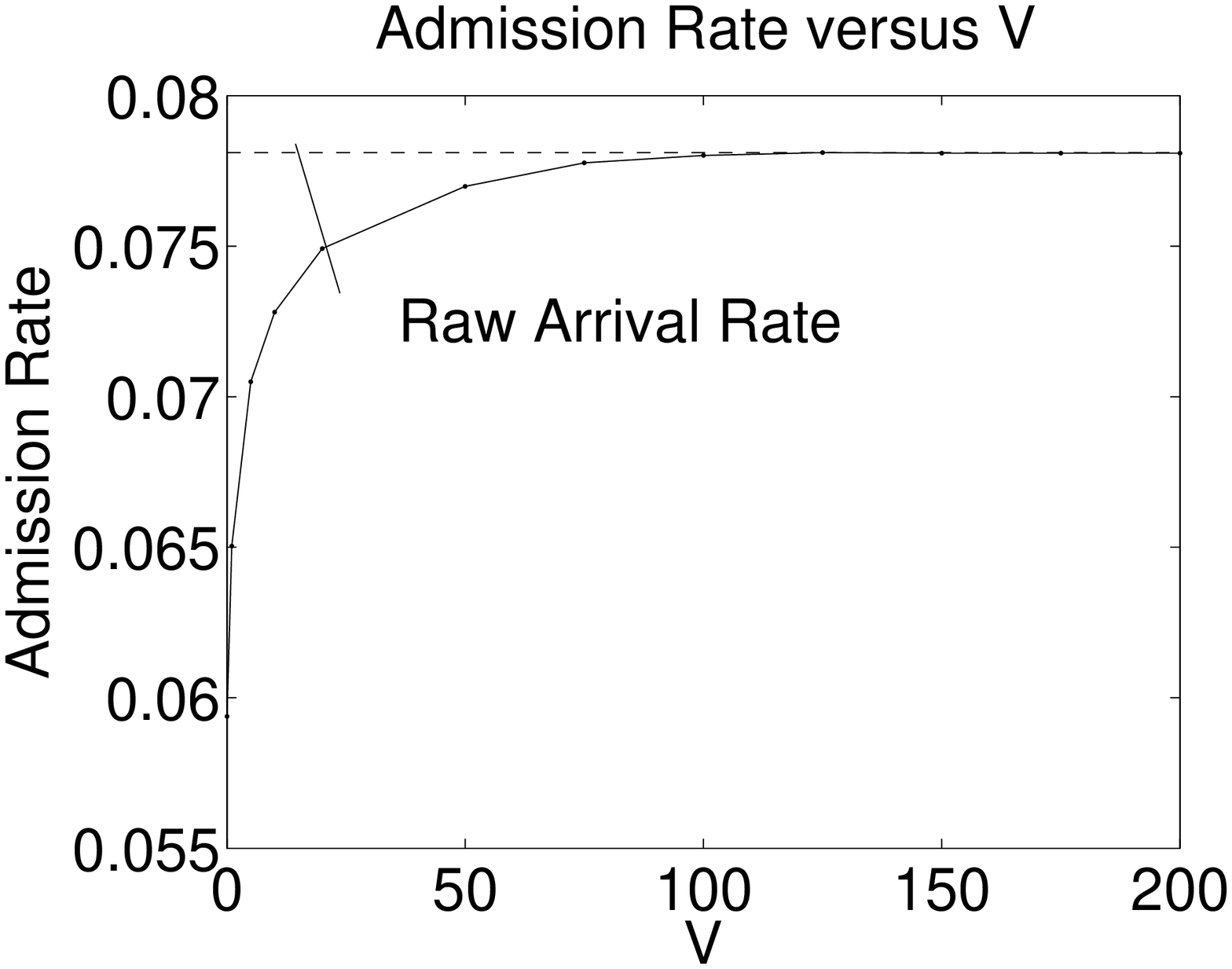}
\caption{Total admission rate versus $V$.}
\label{fig:rand1}
\end{minipage}
\qquad
\qquad
\begin{minipage}{2.6in}
  \includegraphics[height=2.1in, width=2.7in]{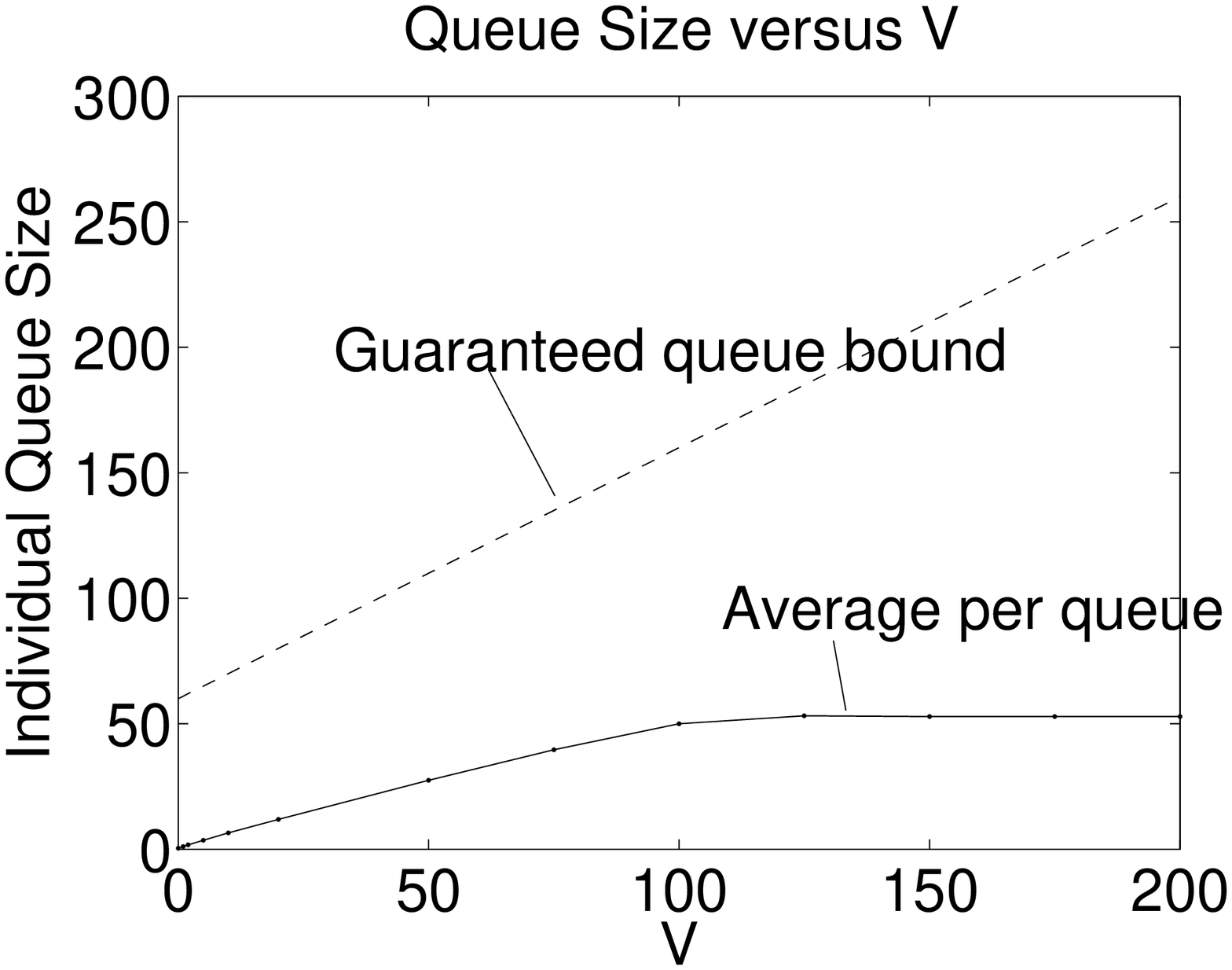} 
  \caption{Average queue size per queue, and the deterministic bound from Section \ref{section:det-bound}.}
  \label{fig:rand2}
\end{minipage}
\end{figure} 

We now illustrate that the algorithm is robust to abrupt rate changes.  We consider the same system
as before, run over 10 million frames with $V=100$. However, we break the simulation timeline into three
equal size phases.  During the first and third phase, 
we use arrival rates $\lambda_i = \rho_1/(30i)$ for $i \in \{1, \ldots, 10\}$, where $\rho_1 = 0.8$ as in the previous experiment. 
Then, during the second (middle) phase, we double the rates to $\lambda_i = \rho_2/(30i)$,
where $\rho_2 = 1.6$.   Because $\rho_2 > 1$, these rates are \emph{infeasible} and the algorithm must
learn to optimally drop tasks so as to maximize the admission rate subject to the power constraint.  Recall
that the algorithm is unaware of the arrival rates and must adapt to the existing system conditions. The 
results are shown in Figs. \ref{fig:rand-vary1} and \ref{fig:rand-vary2}.  Fig. \ref{fig:rand-vary1} shows
a moving average admission rate versus time.  During the first and third phases of the simulation, 
we have $\rho_1 = 0.8$ and the 
admitted rates are close to the raw arrival rate (shown as the lower dashed horizontal line). 
During the middle interval (with $\rho_2 = 1.6$),  the
algorithm quickly adapts to the increased arrivals and yields admitted rates that are close to those
that should be used in a system with loading $\rho_2 = 1.6$ always 
(shown as the higher dashed horizontal line). 
Fig. \ref{fig:rand-vary2} plots the corresponding moving average queue backlog per queue.  The lower 
dashed horizontal 
line indicates the value of average backlog that would be achieved in a system with loading $\rho_1 = 0.8$
always.  Also shown is the deterministic queue bound $V + 60 = 160$, which 
holds for all frames, regardless of the raw arrival rates. 

\begin{figure}[cht]
\centering
\begin{minipage}{2.6in}  
\includegraphics[height=2.1in, width=2.7in]{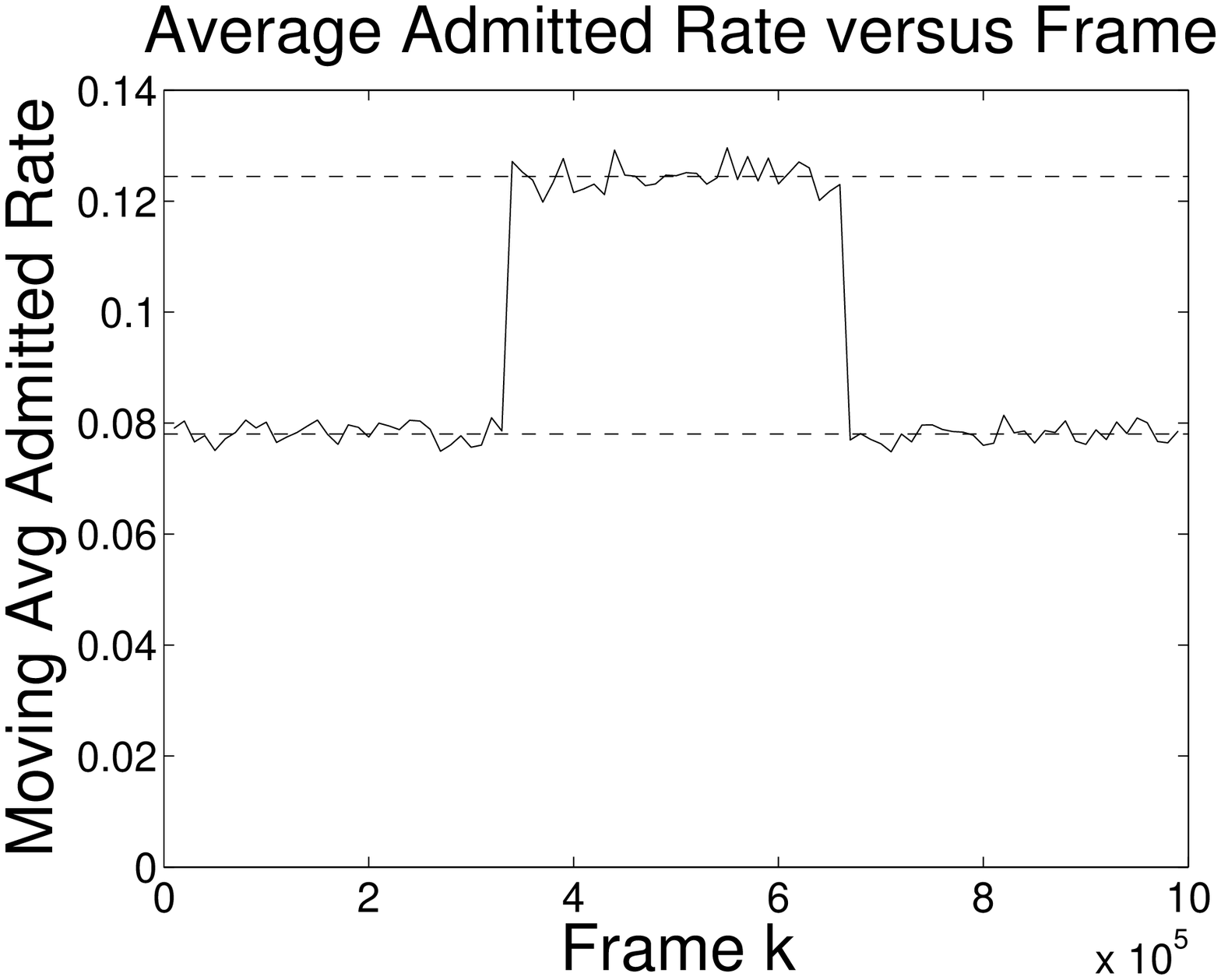}
\caption{Moving average admission rate versus frame index $k$ for the system with abrupt rate changes.}
\label{fig:rand-vary1}
\end{minipage}
\qquad
\qquad
\begin{minipage}{2.6in}
  \includegraphics[height=2.1in, width=2.7in]{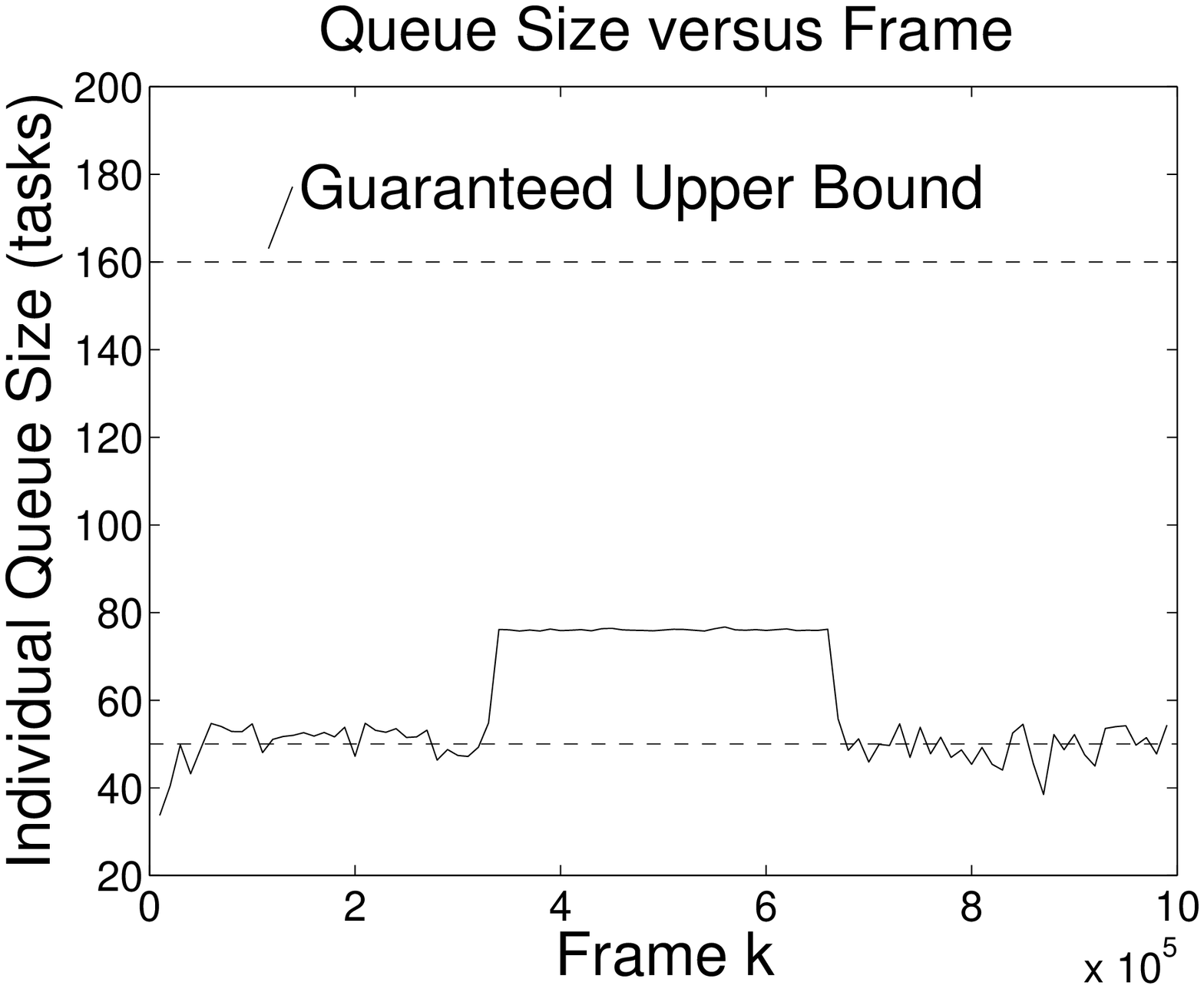} 
  \caption{Individual queue size (average and guaranteed bound) versus frame index $k$ for the system
  with abrupt rate changes.}
  \label{fig:rand-vary2}
\end{minipage}
\end{figure}

\subsection{Task Scheduling: Extensions and Further Reading} 

This section considered minimizing 
time averages subject to time average constraints.  
Extended techniques for minimizing convex functions of time average vectors 
(or maximizing concave functions) subject to similar constraints
are treated in the renewal optimization 
theory of \cite{sno-text}\cite{renewals-asilomar2010}. 
This is useful for extending the flow control problem \eqref{eq:rand1}-\eqref{eq:rand4} to optimize
a sum of concave utility functions $\phi_n(x)$ of the time average admission rates for each class:  
\[ \sum_{n=1}^N\phi_n\left(\overline{A}_n/(\overline{D} + \overline{I})\right) \]
Using logarithmic functions $\phi_n(x)$ leads to \emph{proportional fairness} \cite{kelly-charging},
while other fairness properties are achieved with other 
concave functions \cite{mo-walrand-fair}\cite{low-fairness-properties}\cite{low-non-convex}. 
Related utility optimization for particular classes of wireless systems with fixed size frames of $T$ slots, 
and with probabilistic reception every slot, is treated using a 
different technique in \cite{utility-frame-hou}\cite{frame-hou}. 
Related work on utility optimal scheduling for single-slot data networks is treated using 
convex optimization 
in \cite{kelly-shadowprice}\cite{low-duality}\cite{xiao-johansson-boyd-toc}\cite{lin-shroff-cdc04}, 
stochastic ``primal-dual'' algorithms in \cite{prop-fair-down}\cite{vijay-allerton02}\cite{stolyar-greedy}\cite{atilla-primal-dual-jsac}\cite{primal-dual-cmu}, 
and stochastic ``dual'' algorithms 
in \cite{neely-thesis}\cite{neely-fairness-infocom05}\cite{atilla-fairness}\cite{now}\cite{ribeiro-duality-it}.

For the problem with random arrivals, queueing delay can often be improved by changing
the constraint $\overline{A}_n \leq \overline{1}_n$ of \eqref{eq:rand2} 
to $\overline{A}_n + \epsilon \leq \overline{1}_n$, for some $\epsilon>0$.  This specifies that 
the output processing rate should be $\epsilon$ larger than the arrival rate.  However, unlike the case
$\epsilon=0$, such a constraint yields a queue $Q_n[k]$ that contains some ``fake'' tasks, and can
lead to decisions that serve the queue when no actual tasks are available.  This can be overcome
using the theory of \emph{$\epsilon$-persistent service queues} 
in \cite{neely-smartgrid}\cite{neely-worst-case-delay}.

Note that our flow control algorithm rewards admission of new data on each frame. This 
fits the general framework of this section, 
where attributes $y_l[k]$ are (possibly random) functions of the control 
action $\alpha[k] \in \script{A}$. One might attempt to solve the problem by 
defining a reward upon \emph{completion} of service. 
This does not fit our framework:  That is because the algorithm might ``complete'' a service in a queue that is empty,
simply to accumulate the ``fake'' reward.  One could eliminate fake rewards by an augmented model
that allows rewards $\hat{y}_l(\alpha[k], \bv{Q}[k])$ and/or action spaces $\script{A}(\bv{Q}[k])$ that depend 
on the current backlog, although this creates a much more complex Markov decision model that is easily avoided
by defining rewards at admission. 
However, there are some problems 
that require such a reward structure, such as problems of stock market trading where prior ownership of a 
particular stock is required to reap the benefits of selling at an observed desirable price.  These problems
can be treated with a modified Lyapunov function of the form $L[k] = \sum_{n=1}^N(Q_n[k]-\theta)^2$,
which pushes backlog towards a non-empty state $\theta$.  This approach is used for stock trading
\cite{neely-stock-cdc2010}, inventory control \cite{neely-inventory-control-cdc2010}, energy harvesting
\cite{longbo-energy-harvest-mobihoc}, and smart-grid energy management \cite{rahul-battery}. 
The first algorithms for optimal energy harvesting 
used a related technique \cite{leonidas-energy-harvest}, 
and related problems in processing networks that assemble components by combining different sub-components 
are treated
in \cite{jiang-walrand-book}\cite{longbo-processing-nets-IFIP}. 

\subsection{Exercises for Section \ref{section:general-attributes}}

 \begin{exercise} Consider a system with $N$ classes, each of which always has new tasks available.  
 Every frame $k$ we choose a class $c[k] \in \{1, \ldots, N\}$, and select a single
 task of that class for processing.  We can process using mode $m[k] \in \script{M}$.  The result yields
 frame size $\hat{T}(c[k], m[k])$, energy $\hat{e}(c[k], m[k])$, and \emph{processing quality} $\hat{q}(c[k],m[k])$. 
 Design an algorithm that selects $c[k]$ and $m[k]$ every frame $k$ to  maximize 
 $\overline{q}/\overline{T}$ subject
 to an average power constraint $\overline{e}/\overline{T} \leq P_{av}$ and to a processing rate constraint
 of at least $\lambda_n$ for each class $n$. 
 \end{exercise} 

\begin{exercise} \label{ex:separate-gamma} Show that minimization of \eqref{eq:huge} leads
to the separable flow control and task scheduling algorithm of \eqref{eq:optimal-gamma}-\eqref{eq:other-dec}. 
\end{exercise} 

\begin{exercise} \label{ex:deterministic-queue}  Consider the flow control and task scheduling 
algorithm \eqref{eq:optimal-gamma}-\eqref{eq:other-dec}, and recall that $Q_n[k] \leq Q_{max}$ for all $n \in \{1, \ldots, N\}$ and all frames $k$, where $Q_{max} \defequiv \max_{n\in\{1, \ldots, N\}} [Vw_n + A_{n,max}]$. 
Suppose there is a constant $e_{min}>0$ such 
that $\hat{e}(c,m,I) \geq e_{min}$ for all $c \neq 0$, and that $\hat{e}(0,m,I) = 0$ for all $m \in \script{M}$, $I\in[0, I_{max}]$. 

a) Show that the minimization of \eqref{eq:other-dec} chooses $c[k]=0$ whenever 
$Z[k] > Q_{max}/e_{min}$.  

b)  Suppose there is a constant $e_{max}$ such that $e[k] \leq e_{max}$ for all $k$.  Conclude
that $Z[k] \leq Z_{max} \defequiv Q_{max}/e_{min} + e_{max}$ for all frames $k$. 

c) Use the queueing equation \eqref{eq:h-update} to show that over any sequence of $K$ frames
$\{j, j+1, \ldots, j+K-1\}$ (starting at some frame $j$), the total energy usage satisfies:
\[ \mbox{$\sum_{k=j}^{j+K-1} e[k] \leq P_{av}\sum_{k=j}^{j+K-1}T[k]  + Z_{max}$} \]
where $T[k]$ is the size of frame $k$. Hint:  Make an argument similar to the proof of Lemma \ref{lem:virtual-queues}. 
\end{exercise} 

\begin{exercise} \label{ex:alternate-problem} 
Consider the same general attributes $y_l[k]$ of Section \ref{section:general-attributes}, with 
$\hat{y}_l(\alpha[k])$ for $\alpha[k] \in \script{A}$.  State the algorithm for solving the problem of 
minimizing $\overline{y}_0$ subject to $\overline{y}_l/\overline{T} \leq c_l$.  Hint:  Define appropriate
attributes $x_l[k]$ on a system with an ``effective'' frame size of 1 every frame, and minimize
$\overline{x}_0/1$ subject to $\overline{x}_l/1 \leq 0$ for $l \in \{1, \ldots, L\}$. 
\end{exercise} 

\begin{exercise} Modify the task scheduling algorithm of Section \ref{section:dppr} to allow the controller
to serve more than one task per frame.  Specifically, 
every frame $k$ the controller chooses
a \emph{service action} $s[k]$ from a set $\script{S}$ of possible actions.  Each service action 
$s \in \script{S}$ determines a \emph{clearance vector} $\bv{c}(s) = (c_1(s), \ldots, c_N(s))$, where
$c_i(s)$ is the number of tasks of type $i$ served if action $s$ is used on a given frame. 
It also incurs a delay
$\hat{D}(s)$ and energy $\hat{e}(s)$.  
%State the resulting algorithm. 
%b) Give a specific example of a particular set $\script{S}$ 
%and particular $\bv{c}(s)$, $\hat{D}(s)$, $\hat{e}(s)$ functions. 
%Restate the algorithm using this example. 
\end{exercise} 

\begin{exercise} \label{ex:linear-fractional} Consider the linear fractional problem of finding 
a vector $(x_1, \ldots, x_M)$ to solve: 
\begin{eqnarray*}
\mbox{Minimize:} & \frac{a_0 + \sum_{i=1}^Ma_ix_i}{b_0 + \sum_{i=1}^Mb_ix_i} \\
\mbox{Subject to:} & \sum_{i=1}^M c_{il}x_i \leq d_l \: \: \: \forall l \in \{1, \ldots, L\} \\
& 0 \leq x_i \leq 1 \: \: \: \forall i \in \{1, \ldots, M\} 
\end{eqnarray*}
%minimize the function 
%$(a_0 + \sum_{i=1}^M a_ix_i)/(b_0 + \sum_{i=1}^Mb_ix_i)$ subject to $\sum_{i=1}^M c_{il}x_i \leq d_l$ for $l \in \{1, \ldots, L\}$ and $0 %\leq x_i \leq 1$ for all $i \in \{1, \ldots, M\}$.  
Assume constants $a_i$, $b_i$, $c_{il}$, $d_l$ are given, 
that $b_0>0$, and that $b_i\geq 0$ for all $i \in \{1, \ldots, M\}$.   Treat this as a time average problem 
with action $\alpha[k] = (x_1[k], \ldots, x_M[k])$, action space $\script{A} = \{(x_1, \ldots, x_M) | 0 \leq x_i \leq 1 \: \: \forall i \in \{1, \ldots, M\}\}$, frame size $T[k] = b_0 + \sum_{i=1}^Mb_ix_i[k]$, 
and where we seek to minimize $(a_0 + \sum_{i=1}^Ma_i\overline{x}_i)/(b_0+\sum_{i=1}^Mb_i\overline{x}_i)$ subject to $\sum_{i=1}^M c_{il}\overline{x}_i \leq d_l$ for all $l \in\{1, \ldots, L\}$. 

a) State the drift-plus-penalty ratio algorithm
\eqref{eq:new-min} 
for this, and conclude that: 
\begin{eqnarray*}
 \lim_{K\rightarrow\infty} \frac{a_0 + \sum_{i=1}^Ma_i\overline{x}_i[K]}{b_0 + \sum_{i=1}^Mb_i\overline{x}_i[K]} \leq ratio^{opt} + B/V \: \: \: , \: \: \: 
 \lim_{K\rightarrow\infty} \sum_{i=1}^Mc_{il}\overline{x}_i[K] &\leq& d_l \: \: \: \forall l \in \{1, \ldots, L\} 
\end{eqnarray*}
for some finite constant $B$, 
where $ratio^{opt}$ is the optimal value of the objective function, and
$\overline{x}_i[K] \defequiv \frac{1}{K}\sum_{k=0}^{K-1}x_i[k]$. Thus, 
the limiting time average satisfies the constraints and is within $B/V$ from optimality.  
Your answer should solve for values $\phi_i[k]$ such that on frame $k$ we choose 
$(x_1[k], \ldots, x_M[k])$ over $\script{A}$ 
to minimize $\frac{\phi_0[k] + \sum_{i=1}^M\phi_i[k]x_i[k]}{b_0 + \sum_{i=1}^Mb_ix_i[k]}$.   Note: In Appendix B it is  shown that
this minimization is performed as follows:   Define $\script{I} \defequiv \{i \in \{1, \ldots, M\} | b_i = 0\}$ and $\script{J} \defequiv\{j \in \{1, \ldots, M\} | b_j > 0\}$.  For all $i \in \script{I}$, choose $x_i[k] = 0$ if $\phi_i[k] \geq 0$, and $x_i[k]=1$ if $\phi_i[k] < 0$.  Next, temporarily select $x_j[k]=0$ for all $j \in \script{J}$. Then rank order
the indices $j \in \script{J}$ from smallest to largest value of $\phi_j[k]/b_j$, and, using this order, greedily 
change $x_j[k]$ from $0$ to $1$ if it improves the solution.  We stop when we reach the first index in the rank ordering that
does not improve the solution.

b) Note that the case $b_0=1$ and $b_i=0$ for $i \neq 0$ is a linear program. Give an explicit decision
rule for each $x_i[k]$ in this case (the solution should be separable for each $i\in\{1, \ldots, M\}$). 
\end{exercise}

\section{Reacting to Randomly Observed Events} \label{section:random-events} 

Consider a problem with 
general attributes $y_0[k], y_1[k], \ldots, y_L[k]$ and frame size $T[k]$ for each frame $k$,  as in 
Section \ref{section:general-attributes}.  However, now assume the controller observes
a random event $\omega[k]$ at the beginning of each frame $k$, and this can influence
attributes and frame sizes.  The value of $\omega[k]$ can represent a vector of channel
states and/or prices observed for frame $k$.  Assume $\{\omega[k]\}_{k=0}^{\infty}$ is independent and identically
distributed (i.i.d.) over frames. The controller chooses an action $\alpha[k] \in \script{A}(\omega[k])$,
where the action space $\script{A}(\omega[k])$ possibly depends on the observed $\omega[k]$. 
Values of $(T[k], y_0[k], \ldots, y_L[k])$ are conditionally independent of the past given the 
current $\omega[k]$ and $\alpha[k]$, with mean values: 
\[ \hat{y}_l(\omega[k], \alpha[k]) = \expect{y_l[k]|\omega[k], \alpha[k]}  \: \: \: , \: \: \: \hat{T}(\omega[k], \alpha[k]) = 
\expect{T[k]|\omega[k], \alpha[k]} \]
%Second moments of $T[k]$ and $y_l[k]$ are again assumed to be bounded, 
%regardless of the policy $\alpha[k]$ used.  

The goal is to solve the following optimization problem: 
\begin{eqnarray}
\mbox{Minimize:} && \overline{y}_0/\overline{T} \label{eq:r0}  \\
\mbox{Subject to:} && \overline{y}_l/\overline{T} \leq c_l \: \: \forall l \in \{1, \ldots, L\} \label{eq:r1} \\
&& \alpha[k] \in \script{A}(\omega[k]) \: \: \forall k \in \{0, 1, 2, \ldots\} \label{eq:r2} 
\end{eqnarray}

As before,  the constraints $\overline{y}_l \leq c_l\overline{T}$ are satisfied via virtual queues $Q_l[k]$
for $l \in \{1, \ldots, L\}$: 
\begin{equation} \label{eq:qlast-update} 
Q_l[k+1] = \max[Q_l[k] + y_l[k] - c_lT[k], 0] 
\end{equation} 
The random $\omega[k]$ observations make this problem more complex that those considered in 
previous sections of this paper.   We present
two different algorithms from \cite{sno-text}\cite{renewals-asilomar2010}.
%,
%one that uses the drift-plus-penalty ratio, and a different one that is easier to implement. 

{\bf{Algorithm 1:}} Every frame $k$, observe $\bv{Q}[k]$ and $\omega[k]$ and 
choose $\alpha[k] \in \script{A}(\omega[k])$ to minimize the following ratio of expectations:  
\begin{eqnarray} 
\frac{\expect{V\hat{y}_0(\omega[k],\alpha[k]) + \sum_{l=1}^LQ_l[k]\hat{y}_l(\omega[k], \alpha[k]) |\bv{Q}[k]}}{\expect{\hat{T}(\omega[k],\alpha[k])|\bv{Q}[k]}} \label{eq:alg1} 
\end{eqnarray} 
Then update the virtual queues via \eqref{eq:qlast-update}. 

{\bf{Algorithm 2:}} Define $\theta[0]=0$, and define $\theta[k]$ for $k \in \{1, 2, 3, \ldots\}$ as a 
running ratio of averages over past frames: 
\begin{equation} \label{eq:theta-k} 
\theta[k] \defequiv \mbox{$\sum_{i=0}^{k-1}y_0[i]/\sum_{i=0}^{k-1}T[i]$} 
\end{equation} 
Then every frame $k$, observe $\bv{Q}[k]$, $\theta[k]$, and $\omega[k]$, and 
choose $\alpha[k] \in \script{A}(\omega[k])$ to minimize the following function: 
\begin{equation} \label{eq:alg2} 
 V[\hat{y}_0(\omega[k], \alpha[k]) - \theta[k]\hat{T}(\omega[k],\alpha[k])] + \sum_{l=1}^LQ_l[k][\hat{y}_l(\omega[k], \alpha[k]) - c_l\hat{T}(\omega[k],\alpha[k])]  
 \end{equation} 
Then update the virtual queues via \eqref{eq:qlast-update}. 

\subsection{Comparison of Algorithms 1 and 2} 

Both algorithms are introduced and analyzed in \cite{sno-text}\cite{renewals-asilomar2010}, where
they are shown to satisfy the desired constraints and yield an optimality gap of $O(1/V)$. 
Algorithm 1 can be analyzed in a manner similar to the proof of Theorem \ref{thm:perf1}, and has 
the same tradeoff with $V$ as given in that theorem.  However,  the ratio of expectations \eqref{eq:alg1} 
is \emph{not necessarily} minimized by observing $\omega[k]$ and choosing $\alpha[k]\in\script{A}(\omega[k])$
to minimize the deterministic ratio given $\omega[k]$.   In fact, the minimizing policy 
depends on the (typically unknown) probability distribution for $\omega[k]$.  A more complex \emph{bisection algorithm}
is needed for implementation, as specified in \cite{sno-text}\cite{renewals-asilomar2010}.  

Algorithm 2 is much simpler and involves a greedy selection 
of $\alpha[k] \in \script{A}(\omega[k])$ based only on observation of $\omega[k]$, without requiring knowledge
of the probability distribution for $\omega[k]$. However, its mathematical analysis does not yield
as explicit information regarding convergence time as does Algorithm 1.  Further, it requires a running
average to be kept starting at frame $0$, and hence may not be as adaptive when system statistics change. 
A more adaptive approximation of Algorithm 2 would define the average $\theta[k]$ over a moving window
of some fixed number of frames, or would use an exponentially decaying average. 

Both algorithms reduce to the following simplified \emph{drift-plus-penalty} rule 
in the special case when
the frame size $\hat{T}(\omega[k],\alpha[k])$ is a fixed constant $T$ for all $\omega[k], \alpha[k]$:   
Every frame $k$, observe $\omega[k]$
and $\bv{Q}[k]$ and choose $\alpha[k] \in \script{A}(\omega[k])$ to minimize: 
\begin{equation} \label{eq:fixed-frame} 
 \mbox{$V\hat{y}_0(\omega[k],\alpha[k]) + \sum_{l=1}^LQ_l[k]\hat{y}_l(\omega[k],\alpha[k])$} 
 \end{equation} 
Then update the virtual queues via \eqref{eq:qlast-update}. This special case algorithm was
developed in \cite{neely-energy-it} to treat systems with fixed size time slots. 

A simulation comparison of the algorithms is given in \cite{renewals-asilomar2010}. The next
subsection describes an application to energy-aware computation
and transmission in a wireless smart phone.  Exercise \ref{ex:opportunistic-scheduling}
considers opportunistic scheduling where wireless transmissions can be deferred by waiting for more
desirable channels. Exercise \ref{ex:energy-price} considers an example of 
price-aware energy consumption for a network server that can process computational tasks or 
outsource them to another server. 

\subsection{Efficient Computation and Transmission for a Wireless Smart Device} \label{section:smartphone} 

Consider a wireless smart device (such as a smart phone or sensor) that always has tasks to process. 
Each task involves a computation operation, followed by a transmission operation over a wireless
channel. 
On each frame $k$, the device takes a new task and looks at its \emph{meta-data} $\beta[k]$, being information
that characterizes the task in terms of its computational and transmission requirements.  Let $d$ represent
the time required to observe this meta-data. The device then chooses a \emph{computational processing
mode} $m[k] \in\script{M}(\beta[k])$, where $\script{M}(\beta[k])$ is the set of all mode options under 
$\beta[k]$.  The mode $m[k]$ and meta-data $\beta[k]$ affect a computation time $D_{comp}[k]$, 
computation energy $e_{comp}[k]$, computation quality $q[k]$, 
and generate $A[k]$ bits for transmission over the channel.
The expectations of $D_{comp}[k]$, $e_{comp}[k]$, $q[k]$ are: 
\begin{eqnarray*}
\hat{D}_{comp}(\beta[k], m[k]) &=& \expect{D_{comp}[k] | \beta[k], m[k]} \\
 \hat{e}_{comp}(\beta[k],m[k]) &=& \expect{e_{comp}[k]|\beta[k], m[k]} \\
 \hat{q}(\beta[k], m[k]) &=& \expect{q[k]|\beta[k],m[k]} 
  \end{eqnarray*}
  For example, in a wireless sensor, the mode $m[k]$ may represent a particular sensing task, where
  different tasks can have different qualities and thus incur 
  different energies, times, and $A[k]$ bits for transmission.  The full conditional distribution of $A[k]$, 
  given $\beta[k]$, $m[k]$, will play a role in the transmission stage (rather than just its conditional expectation). 

The $A[k]$ units of data must be transmitted over a wireless channel.  Let $S[k]$ be the state of the 
channel on frame $k$, and assume $S[k]$ is constant for the duration of the frame.  We choose
a \emph{transmission mode} $g[k] \in \script{G}$, yielding a transmission time $D_{tran}[k]$
and transmission energy $e_{tran}[k]$ with expectations that depend on $S[k]$, $g[k]$, and $A[k]$.
Define random event $\omega[k] = (\beta[k], S[k])$ and action $\alpha[k] = (m[k], g[k])$.  
We can then define expectation functions $\hat{D}_{tran}(\omega[k], \alpha[k])$ and 
$\hat{e}_{tran}(\omega[k], \alpha[k])$ by: 
\[ \hat{D}_{tran}(\omega[k], \alpha[k]) = \expect{D_{tran}[k]|\omega[k], \alpha[k]} \: \: \: , \: \: \: \hat{e}_{tran}(\omega[k], \alpha[k]) = \expect{e_{tran}[k]|\omega[k], \alpha[k]} \]
where the above expectations are defined via the  
\emph{conditional distribution} associated with the number of bits $A[k]$
at the computation output,  given the 
meta-data $\beta[k]$ and computation mode  $m[k]$ selected in the computation phase (where
$\beta[k]$  and $m[k]$ are included in the $\omega[k]$, $\alpha[k]$ information). 
The total frame size is thus $T[k] = d + D_{comp}[k] + D_{tran}[k]$.

The goal is to maximize  frame processing quality per unit time $\overline{q}/\overline{T}$ subject to 
a processing rate constraint of $1/\overline{T} \geq \lambda$, 
and subject to an average power constraint 
$(\overline{e}_{comp} + \overline{e}_{tran})/\overline{T} \leq P_{av}$ (where $\lambda$ and $P_{av}$ are given 
constants). 
This fits the general framework with observed random events $\omega[k] = (\beta[k], S[k])$, control actions 
$\alpha[k] = (m[k], g[k])$, and action space $\script{A}(\omega[k]) = \script{M}(\beta[k]) \times \script{G}$. 
We can define $y_0[k] = -q[k]$, $y_1[k] = T[k] - 1/\lambda$,  and $y_2[k] = e_{comp}[k] + e_{tran}[k] - P_{av}T[k]$,
and solve the problem of minimizing $\overline{y}_0/\overline{T}$ subject to $\overline{y}_1 \leq 0$
and $\overline{y}_2 \leq 0$. 
To do so, let $Q[k]$ and $Z[k]$ be virtual queues for the two constraints: 
\begin{eqnarray}
Q[k+1] &=& \max[Q[k] + T[k] - 1/\lambda, 0] \label{eq:smart-q-update} \\
Z[k+1] &=& \max[Z[k] + e_{comp}[k] + e_{tran}[k] - P_{av}T[k], 0] \label{eq:smart-z-update} 
\end{eqnarray}
Using Algorithm 2, we define $\theta[0]=0$ and $\theta[k]$ for $k \in \{1, 2,\ldots\}$ by \eqref{eq:theta-k}.
The Algorithm 2 minimization \eqref{eq:alg2} amounts to observing $(\beta[k], S[k])$, $Q[k]$, $Z[k]$, and 
$\theta[k]$ on each frame $k$  and choosing
$m[k]\in\script{M}(\beta[k])$ and $g[k] \in \script{G}$ to minimize:  
\begin{eqnarray*} 
V[-\hat{q}(\beta[k],m[k]) - \theta[k](d + \hat{D}_{comp}(\beta[k],m[k]) + \hat{D}_{tran}(\omega[k],\alpha[k]))] \\
+  Q[k][d + \hat{D}_{comp}(\beta[k],m[k]) + \hat{D}_{tran}(\omega[k], \alpha[k]) - 1/\lambda] \\
+ Z[k][\hat{e}_{comp}(\beta[k],m[k]) + \hat{e}_{tran}(\omega[k],\alpha[k]) - P_{av}(d + \hat{D}_{comp}(\beta[k],m[k]) + 
\hat{D}_{tran}(\omega[k], \alpha[k]))]
\end{eqnarray*}
The computation and transmission operations are coupled and cannot be separated. 
This yields the following algorithm: Every frame $k$: 
\begin{itemize} 
\item Observe $\omega[k] = (\beta[k], S[k])$ and values $Q[k]$, $Z[k]$, $\theta[k]$.
Then \emph{jointly} choose  action $m[k] \in \script{M}(\beta[k])$ and $g[k] \in \script{G}$, for a combined
action $\alpha[k] = (m[k], g[k])$,  to minimize: 
\begin{eqnarray*}
-V\hat{q}(\beta[k], m[k]) + \hat{D}_{comp}(\beta[k], m[k])[Q[k] -V\theta[k] -P_{av}Z[k]] + Z[k]\hat{e}_{comp}(\beta[k],m[k]) \\
+   \hat{D}_{tran}(\omega[k], \alpha[k])[Q[k] -V\theta[k]  - P_{av}Z[k]] + Z[k]\hat{e}_{tran}(\omega[k],\alpha[k]) 
\end{eqnarray*}

\item (Updates) Update $Q[k]$, $Z[k]$, $\theta[k]$ via \eqref{eq:smart-q-update}, \eqref{eq:smart-z-update}, 
and \eqref{eq:theta-k}. 
\end{itemize} 

Exercise \ref{ex:opt-q} shows the algorithm can be implemented without $\theta[k]$ if
the goal is changed to maximize $\overline{q}$, rather than $\overline{q}/\overline{T}$. 
Further, the computation and transmission decisions can be \emph{separated} if the system is modified so that 
the bits generated from computation are handed to a separate transmission layer 
for eventual transmission over the channel, rather than requiring transmission on the same frame, 
similar to the structure used in Exercise \ref{ex:energy-price}.

\subsection{Exercises for Section \ref{section:random-events}} 

\begin{exercise} \label{ex:opportunistic-scheduling} (Energy-Efficient 
Opportunistic Scheduling \cite{neely-energy-it}) 
Consider a wireless device that operates over fixed size time slots $k \in \{0, 1, 2, \ldots\}$. 
Every slot $k$, new data of size $A[k]$ bits arrives and is added to a queue.   
The data must eventually be transmitted over a time-varying channel.  
At the beginning of every slot $k$, a controller observes the \emph{channel state} $S[k]$
and allocates power $p[k] $ for transmission, enabling transmission of $\mu[k]$ bits, 
where  $\mu[k] = \hat{\mu}(S[k], p[k])$ for some given function $\hat{\mu}(S,p)$. 
Assume $p[k]$ is chosen so that $0 \leq p[k] \leq P_{max}$ for some 
constant $P_{max}>0$. We want to minimize average power $\overline{p}$ subject to 
supporting all data, so that $\overline{A} \leq \overline{\mu}$. 
Treat this as a problem with all frames equal to 1 slot, observed
random events $\omega[k] = (A[k], S[k])$, and actions $\alpha[k] = p[k] \in [0, P_{max}]$.  Design an appropriate
queue update and power allocation algorithm, using the policy structure \eqref{eq:fixed-frame}. 
\end{exercise} 

\begin{figure}[htbp]
   \centering
   \includegraphics[height=1.5in, width=5.5in]{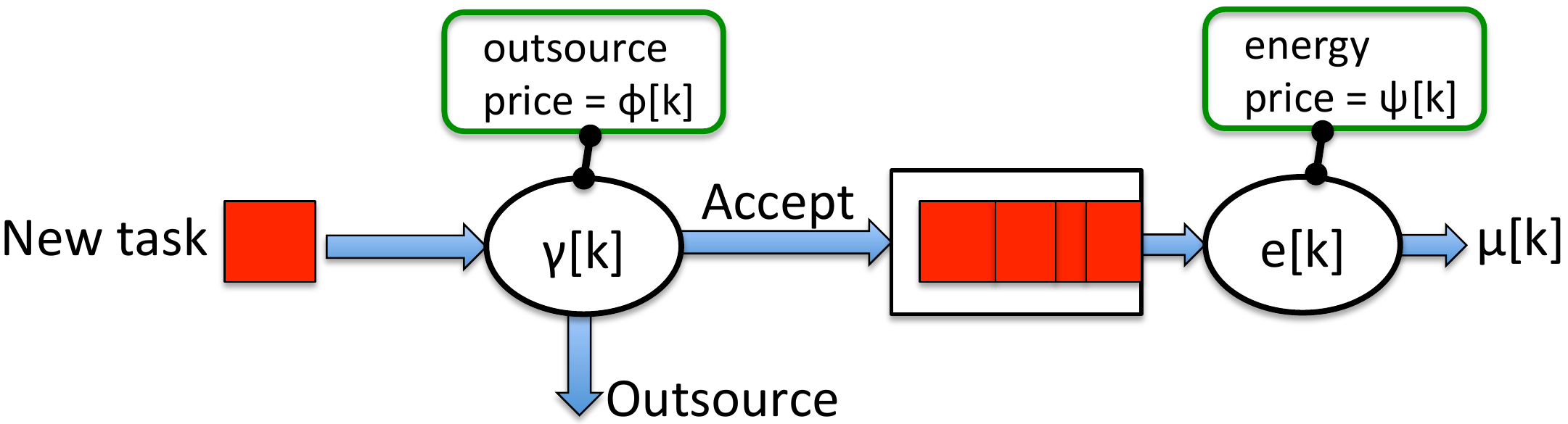} % requires the graphicx package
   \caption{The client/server system of Exercise \ref{ex:energy-price}, with decision variables $(\gamma[k], e[k])$ and 
   observed prices $(\phi[k], \psi[k])$.}
   \label{fig:energy-price}
\end{figure}

\begin{exercise} \label{ex:energy-price} (Energy Prices and Network Outsourcing) 
Consider a computer server that operates over fixed length time slots $k \in \{0, 1, 2,\ldots\}$. 
Every slot $k$, a new task arrives and has size $S[k]$ (if no task 
arrives then $S[k]=0$).  
The server decides to either accept the task, or outsource it to another server (see Fig. \ref{fig:energy-price}).  
Let $\gamma[k]$ be a binary decision variable that is 1 if the server accepts the task on slot $k$,
and zero else. Define $A[k] = \gamma[k]S[k]$ as the total workload admitted on slot $k$, which 
is added to the queue of work to be done. 
Let $\phi[k]$ be the (possibly time-varying) cost
per unit size for outsourcing, so the outsourcing cost is $c_{out}[k] = \phi[k](1-\gamma[k])S[k]$. 
Every slot $k$, the server additionally 
decides to process some of its backlog by purchasing an amount of energy $e[k]$ at a per-unit 
energy price $\psi[k]$, serving $\mu[k] = \hat{\mu}(e[k])$ units of backlog with cost $c_{energy}[k] = \psi[k]e[k]$, where $\hat{\mu}(e)$ is 
some given function. 
Assume $e[k]$ is chosen in some interval $0 \leq e[k] \leq e_{max}$. 
The goal is to minimize time average cost $\overline{c}_{out} + \overline{c}_{energy}$ subject to 
supporting all tasks, so that $\overline{A} \leq \overline{\mu}$.  Treat this as a problem with all frames
equal to 1 slot, observed random events $\omega[k] = (S[k], \phi[k], \psi[k])$, 
actions $\alpha[k] = (\gamma[k], e[k])$, and action space $\script{A}(\omega[k])$ being the
set of all $(\gamma, e)$ such that $\gamma \in \{0,1\}$ and $0 \leq e \leq e_{max}$. Design an appropriate
queue and state the dynamic
algorithm \eqref{eq:fixed-frame} 
for this problem. Show that it can be implemented without knowledge of the probability distribution 
for $(S[k], \phi[k], \psi[k])$, and that the $\gamma[k]$ and $e[k]$ decisions are separable. 
\end{exercise} 

\begin{exercise} \label{ex:opt-q} 
Consider the same problem of Section \ref{section:smartphone}, with the modified objective of 
maximizing $\overline{q}$ subject to $1/\overline{T} \geq \lambda$ and 
$(\overline{e}_{comp} + \overline{e}_{tran})/\overline{T} \leq P_{av}$. Use the observation  
from Exercise \ref{ex:alternate-problem} that this can be viewed as a problem with an ``effective''
fixed frame size equal to 1, and give an algorithm from the policy structure \eqref{eq:fixed-frame}. 
\end{exercise} 

\section{Conclusions} 

This paper presents a methodology for optimizing time averages in systems with variable length
frames.  Applications include energy and quality aware task scheduling in smart
phones, cost effective energy management at computer servers, 
and more.  The resulting algorithms are dynamic and often do not require knowledge of the probabilities
that affect system events.    While the performance theorem of this paper was stated under simple 
i.i.d. assumptions, 
the same algorithms are often provably robust to non-i.i.d. situations, including
situations where the events are non-ergodic \cite{sno-text}. 
Simulations in Section \ref{section:simulation-random} show examples of how such algorithms
adapt to changes in the event probabilities.  Exercises in this paper were included to help readers 
learn to design dynamic algorithms for their own optimization problems.

The solution technique of this paper uses the theory of optimization for
renewal systems from \cite{sno-text}\cite{renewals-asilomar2010}, which  
applies to general problems. 
Performance for individual problems can often be improved using enhanced techniques, 
such as using place-holder backlog, exponential Lyapunov functions, LIFO scheduling, and 
$\epsilon$-persistent
service queues, as discussed in \cite{sno-text} and references therein.  However, 
the drift-plus-penalty methodology described in this paper provides much of the insight
needed for these more advanced 
techniques. It is also simple to work with and typically yields desirable solution 
quality and convergence time.

%%%%%%%%%%%%%%%%%%%%%%%%%%%%%%
%%%%%%%%%%%%%%%%%%%%%%%%%%%%%%%%

\section*{Appendix A --- Bounded Moment Convergence Theorem} 

This appendix provides a \emph{bounded moment convergence theorem} that is often 
more convenient than the standard Lebesgue dominated convergence theorem (see, for example, 
\cite{williams-martingale} for the standard Lebesgue dominated convergence theorem).  
We are unaware of a statement and proof in the literature, so
we give one here for completeness.   Let $X(t)$ be a random process defined either over
non-negative real numbers $t \geq 0$, or discrete time $t \in \{0, 1, 2, \ldots\}$.  Recall that 
$X(t)$ converges \emph{in probability} to a constant $x$ if for all $\epsilon>0$ we have: 
\[ \lim_{t\rightarrow\infty} Pr[|X(t)-x|>\epsilon] = 0 \]

\begin{thm} \label{thm:bmc} Suppose there is a real number $x$ such that $X(t)$ converges
to $x$ in probability. 
Further suppose there are finite constants $C>0$, $\delta>0$ such that for all $t$ we have
$\expect{|X(t)|^{1+\delta}} \leq C$. 
Then $\lim_{t\rightarrow\infty} \expect{X(t)} = x$. 
\end{thm}  

\begin{proof} 
Without loss of generality, assume $x=0$ (else, we can define $Y(t) = X(t) - x$). Fix $\epsilon>0$. 
By definition of $X(t)$ converging to $0$ in probability, we have: 
\begin{equation} \label{eq:in-prob} 
 \lim_{t\rightarrow\infty} Pr[|X(t)|>\epsilon] = 0 
 \end{equation} 
Further, for all $t$ we have $\expect{X(t)} \leq \expect{|X(t)|}$, and so: 
\begin{equation} \label{eq:bound1} 
\expect{X(t)} \leq \epsilon + \expect{|X(t)| \left|\right. |X(t)|>\epsilon}Pr[|X(t)|>\epsilon] 
\end{equation} 
We want to show the final term in the right-hand-side above converges to $0$ when $t\rightarrow\infty$. 
To this end, note that for all $t$ we have: 
\begin{eqnarray}
C &\geq& \expect{|X(t)|^{1+\delta}} \nonumber \\
&\geq& \expect{|X(t)|^{1+\delta} \left|\right. |X(t)|>\epsilon}Pr[|X(t)|>\epsilon] \nonumber \\
&\geq& \expect{|X(t)| \left|\right. |X(t)|>\epsilon}^{1+\delta} Pr[|X(t)|>\epsilon] \label{eq:app-jensen} 
\end{eqnarray}
where \eqref{eq:app-jensen} follows by Jensen's inequality applied to the conditional expectation
of the function $f(|X(t)|)$, where $f(x) = x^{1+\delta}$ is convex over $x \geq 0$.  Multiplying inequality 
\eqref{eq:app-jensen} by $Pr[|X(t)|>\epsilon]^{\delta}$ yields: 
\[ C Pr[|X(t)|>\epsilon]^{\delta} \geq (\expect{|X(t)| \left|\right. |X(t)|>\epsilon}Pr[|X(t)|>\epsilon])^{1+\delta} \geq 0 \]
Taking a limit of the above as $t\rightarrow\infty$ and using \eqref{eq:in-prob} yields: 
\[  0 \geq \lim_{t\rightarrow\infty} (\expect{|X(t)| \left|\right. |X(t)|>\epsilon}Pr[|X(t)|>\epsilon])^{1+\delta} \geq 0 \]
It follows that: 
\[ \lim_{t\rightarrow\infty} \expect{|X(t)|  \left|\right.  |X(t)|>\epsilon}Pr[|X(t)|>\epsilon] = 0 \]
Using this equality and taking a $\limsup$ of \eqref{eq:bound1} yields 
$\limsup_{t\rightarrow\infty} \expect{X(t)} \leq \epsilon$. 
This holds for all $\epsilon >0$, and so $\limsup_{t\rightarrow\infty} \expect{X(t)} \leq 0$. 
Similarly, it can be shown that $\liminf_{t\rightarrow\infty} \expect{X(t)} \geq 0$. 
Thus, $\lim_{t\rightarrow\infty} \expect{X(t)} = 0$. 
\end{proof} 

Recall that convergence with probability 1 is stronger than convergence in probability, and so 
the above result also holds if $\lim_{t\rightarrow\infty} X(t) = x$ with probability 1.  Theorem \ref{thm:bmc} 
can be applied to the case when $\lim_{K\rightarrow\infty} \frac{1}{K}\sum_{k=0}^{K-1} e[k] = \overline{e}$
with probability 1, for some finite constant $\overline{e}$, and when there is a constant $C$ such that 
$\expect{e[k]^2} \leq C$ for all $k$.  Indeed, one can define $X(K) = \frac{1}{K}\sum_{k=0}^{K-1} e[k]$ for
$K\in\{1, 2, 3, \ldots\}$
and use the Cauchy-Schwartz inequality to show
$\expect{X(K)^2} \leq C$ for all $K \in \{1, 2, 3, \ldots\}$, and so 
$\lim_{K\rightarrow\infty} \frac{1}{K}\sum_{k=0}^{K-1}\expect{e[k]} = \overline{e}$. 

\section*{Appendix B --- The Decision Rule for Linear Fractional Programming} 

This appendix proves optimality of the decision rule for online 
linear fractional programming given in Exercise \ref{ex:linear-fractional}. 
Specifically, suppose we have real numbers $\phi_0, \phi_1, \ldots, \phi_M$,
a positive real number $b_0>0$, and non-negative real numbers $b_1, \ldots, b_M$. 
The goal is to solve the following problem: 
\begin{eqnarray*}
\mbox{Minimize:} & \frac{\phi_0 + \sum_{i=1}^M\phi_ix_i}{b_0 + \sum_{i=1}^Mb_ix_i} \\
\mbox{Subject to:} & 0 \leq x_i \leq 1 \: \: \forall i \in \{1, \ldots, M\}
\end{eqnarray*}
%Let $val^*$ denote the minimum value of the objective function 
%for the problem \eqref{eq:appb-p0}-\eqref{eq:appb-p1}.  

To this end, define disjoint 
sets $\script{I}$ and $\script{J}$ that partition the indices $\{1, \ldots, M\}$ as follows: 
\[ \script{I} \defequiv \{i \in \{1, \ldots, M\} | b_i = 0\} \: \: , \: \: \script{J} \defequiv \{j \in \{1, \ldots, M\} | b_j > 0\} \]
Then the objective function can be written: 
\[ \frac{\phi_0 + \sum_{i=1}^M\phi_ix_i}{b_0  + \sum_{i=1}^Mb_ix_i}  = \frac{\phi_0 + \sum_{i\in\script{I}} \phi_ix_i}{b_0 + \sum_{j\in\script{J}} b_jx_j}
+ \frac{\sum_{j\in\script{J}} \phi_jx_j}{b_0 + \sum_{j\in\script{J}} b_jx_j}\]
The denominator is always positive, and so regardless of the $x_j$ values for $j \in \script{J}$, the first term on the right-hand-side
above 
is minimized by choosing $x_i$ for all $i \in \script{I}$ such that $x_i = 0$ if $\phi_i \geq 0$, and $x_i = 1$ if $\phi_i < 0$. 
Let $x_i^*$ for $i \in \script{I}$ denote these decisions. 
It remains only to choose optimal $x_j$ values for $j \in \script{J}$.   Define $\gamma_0 \defequiv \phi_0 + \sum_{i\in\script{I}} \phi_ix_i^*$. 
We want to solve the following: 
\begin{eqnarray}
\mbox{Minimize:} & \frac{\gamma_0 + \sum_{j\in\script{J}} \phi_jx_j}{b_0 + \sum_{j \in \script{J}} b_jx_j} \label{eq:appb-p0} \\
\mbox{Subject to:} & 0 \leq x_j \leq 1 \: \: \forall j \in \script{J} \label{eq:appb-p1} 
\end{eqnarray}
Define $val^*$ as the minimum value for the objective function in the problem \eqref{eq:appb-p0}-\eqref{eq:appb-p1}. 

\begin{lem} \label{lem:rank-order}
The following values $(x_j^*)_{j\in\script{J}}$ solve the problem \eqref{eq:appb-p0}-\eqref{eq:appb-p1}: 
For each $j \in \script{J}$, we have: 
\begin{equation} \label{eq:xjstar} 
 x_j^* = \left\{ \begin{array}{ll}
                          0  &\mbox{ if $\phi_j/b_j \geq val^*$} \\
                             1  & \mbox{ if $\phi_j/b_j < val^*$} 
                            \end{array}
                                 \right.
                                 \end{equation} 
\end{lem} 
\begin{proof} 
Let $(x_j)_{j\in\script{J}}$ be an optimal solution to \eqref{eq:appb-p0}-\eqref{eq:appb-p1}, so that $0 \leq x_j \leq 1$
for all $j \in \script{J}$ and: 
\[ val^* = \frac{\gamma_0 + \sum_{j\in\script{J}} \phi_jx_j}{b_0 + \sum_{j \in \script{J}} b_jx_j}  \]
Define $c$ and $d$ by: 
\[ c \defequiv \gamma_0 + \sum_{j\in\script{J}} \phi_jx_j \: \: , \: \: d \defequiv b_0 + \sum_{j\in\script{J}} b_jx_j \]
so that $val^* = c/d$.  
Now take any $k \in \script{J}$.  It suffices so show we can change $x_k$ to $x_k^*$, as defined in 
\eqref{eq:xjstar}, without changing the ratio from its value $val^*$.  Indeed, if this is true, then 
we can sequentially change each component
$x_j$ to $x_j^*$ without changing the ratio, and so $(x_j^*)_{j\in\script{J}}$ is also 
an optimal solution. 

To this end, define $val$ as the corresponding ratio with $x_k$ replaced with $x_k^*$.   We show that $val=val^*$. 
Clearly $val \geq val^*$ by definition of $val^*$ as the minimum ratio.  To show the opposite inequality, 
define $\delta_k^* = x_k^* - x_k$, and 
note that: 
\begin{eqnarray}
val &=& \frac{c + \phi_k\delta_k^*}{d + b_k\delta_k^*} \nonumber \\
&=& \frac{c}{d} +  \frac{c + \phi_k\delta_k^*}{d + b_k\delta_k^*} - \frac{c}{d} \nonumber \\
&=& \frac{c}{d} + \frac{cd + \phi_k\delta_k^*d - cd - cb_k\delta_k^*}{d(d+b_k\delta_k^*)} \nonumber \\
&=& \frac{c}{d} + \delta_k^*[\phi_k/b_k - c/d]\frac{b_k}{d+b_k\delta_k^*} \nonumber \\
&=& val^* + \delta_k^*[\phi_k/b_k - val^*]\frac{b_k}{d+b_k\delta_k^*} \label{eq:appb-last} 
\end{eqnarray}
Note that $b_k/(d+b_k\delta_k^*)\geq0$, and hence the second term on the right-hand-side of
\eqref{eq:appb-last} is non-positive if $\delta_k^*[\phi_k/b_k - val^*] \leq 0$.
In the case $\phi_k/b_k \geq val^*$, then by \eqref{eq:xjstar} we have $x_k^* = 0$. Thus, $\delta_k^* \leq 0$ and
$\delta_k^*[\phi_k/b_k-val^*] \leq 0$.  Thus, the second term in the right-hand-side of \eqref{eq:appb-last} is non-positive and 
so  $val \leq val^*$.  In the opposite case $\phi_k/b_k < val^*$, then $x_k^* = 1$ and so $\delta_k^*\geq0$, 
$\delta_k^*[\phi_k/b_k-val^*] \leq 0$. Thus, the second term on the right-hand-side of \eqref{eq:appb-last}  is 
again non-positive, so that $val \leq val^*$. 
\end{proof} 

Now rank order the indices $j \in \script{J}$ from smallest to largest value of 
$\phi_j/b_j$, breaking ties arbitrarily. The solution of Lemma \ref{lem:rank-order}
has the form where $x_j^*=1$ for the first $n$ indices in this rank order, and $x_j^*=0$ 
for the remaining indices, 
for some value $n \in \{0, 1, \ldots, |\script{J}|\}$ (where $|\script{J}|$ is the size of set $\script{J}$). 
Define $val_n$ as the value of the ratio when $x_j=1$ for the first $n$ indices in the rank order, and $x_j=0$
for the remaining indices in $\script{J}$.  Start with $val_0=\gamma_0/b_0$. 
According to the rank ordering of $\script{J}$, successively change $x_j$ from $0$ to $1$ if it strictly decreases
the ratio, computing $val_1$, $val_2$, and so on, until we reach the first index that does  
\emph{not} improve the ratio.  By
an argument similar to \eqref{eq:appb-last}, it is easy to see 
that this can only occur at an index $j$ such that $\phi_j/b_j$ is greater
than or equal to the current ratio, which means it is also greater than or equal to $val^*$ (since by definition,  
$val^*$ is less than or equal to any achievable ratio). Thus, all indices $i\in\script{J}$ that come after $j$ in the rank ordering 
also have $\phi_i/b_i \geq val^*$, and this greedy approach has thus arrived at the solution \eqref{eq:xjstar}. 

% ------------------------------------------------------------------------
%GATHER{Xbib.bib}   % For Gather Purpose Only
%GATHER{Paper.bbl}  % For Gather Purpose Only
\bibliographystyle{unsrt}
\bibliography{../../../../latex-mit/bibliography/refs}
\end{document}